\documentclass[12pt, reqno]{amsart}
\usepackage{amssymb,latexsym,amsmath,amscd,amsthm,graphicx, color}
\usepackage[all]{xy}
\usepackage{pgf,tikz}
\usepackage{mathrsfs}
\usetikzlibrary{arrows}
\usepackage{pgfplots}
\pgfplotsset{compat=1.15}
\usepackage[left=0.6 in, top=0.6 in, right=0.6 in, bottom=0.3 in]{geometry}
\usepackage{changepage}
\allowdisplaybreaks
\makeatletter
\newcommand{\setword}[2]{%
	\phantomsection
	#1\def\@currentlabel{\unexpanded{#1}}\label{#2}%
}
\makeatother

\usepackage[linkcolor=blue, urlcolor=blue, citecolor=blue,
colorlinks, bookmarks]{hyperref}

\definecolor{uuuuuu}{rgb}{0.26666666666666666,0.26666666666666666,0.26666666666666666}
\definecolor{xdxdff}{rgb}{0.49019607843137253,0.49019607843137253,1.}
\definecolor{ffqqqq}{rgb}{1.,0.,0.}
\definecolor{ffqqqq}{rgb}{1.,0.,0.}
\definecolor{ffxfqq}{rgb}{1.,0.4980392156862745,0.}

\definecolor{uuuuuu}{rgb}{0.26666666666666666,0.26666666666666666,0.26666666666666666}
\definecolor{qqwuqq}{rgb}{0.,0.39215686274509803,0.}
\definecolor{zzttqq}{rgb}{0.6,0.2,0.}
\definecolor{xdxdff}{rgb}{0.49019607843137253,0.49019607843137253,1.}
\definecolor{qqqqff}{rgb}{0.,0.,1.}
\definecolor{cqcqcq}{rgb}{0.7529411764705882,0.7529411764705882,0.7529411764705882}
\definecolor{sqsqsq}{rgb}{0.12549019607843137,0.12549019607843137,0.12549019607843137}
\definecolor{uuuuuu}{rgb}{0.26666666666666666,0.26666666666666666,0.26666666666666666}
\definecolor{ffqqqq}{rgb}{1,0,0}
\definecolor{xdxdff}{rgb}{0.49019607843137253,0.49019607843137253,1}

\definecolor{yqqqyq}{rgb}{0.5019607843137255,0,0.5019607843137255}
\definecolor{qqqqff}{rgb}{0,0,1}
\definecolor{ffqqqq}{rgb}{1,0,0}
\definecolor{ffqqff}{rgb}{1,0,1}

\theoremstyle{plain}

\newtheorem{theorem}[subsection]{Theorem}

\newtheorem{prop}[subsection]{Proposition}
\newtheorem{prop1}[subsubsection]{Proposition}

\theoremstyle{definition}
\newtheorem{defi1}[subsection]{Definition}
\newtheorem{cor}[subsection]{Corollary}
\newtheorem{exam}[subsection]{Example}
\newtheorem{fact}[subsection]{Fact}
\newtheorem{remark}[subsection]{Remark}
\newtheorem{remark1}[subsubsection]{Remark}

\newtheorem{note}[subsection]{Note}

\newcommand{\uu}{\cup}
\newcommand{\ii}{\cap}
\newcommand{\UU}{\bigcup}
\newcommand{\II}{\bigcap}

\newcommand{\ci}{\subseteq}

\newcommand{\es}{\emptyset}
\newcommand{\set}[1]{\{#1\}}

\newcommand{\ga}{\alpha}

\renewcommand{\gg}{\gamma}

\newcommand{\gk}{\kappa}

\newcommand{\gq}{\theta}

\newcommand{\tit}{\textit}

\newcommand{\C}[1]{\mathcal{#1}}
\newcommand{\D}[1]{\mathbb{#1}}

\newcommand{\te}{\text}

\newcommand{\ol}{\overline}
\newcommand{\ul}{\underline}

\newcommand{\nd}{\noindent}

\newcommand{\pa}{\partial}

\begin{document}

  \nd To appear,\tit{ Mathematics} 
 
	\title{Constrained quantization for probability distributions}

	\author{$^1$Megha Pandey}
	\author{$^2$Mrinal Kanti Roychowdhury}

	\address{$^{1}$Department of Mathematical Sciences \\
		Indian Institute of Technology (Banaras Hindu University)\\
		Varanasi, 221005, India.}
	\address{$^{2}$School of Mathematical and Statistical Sciences\\
		University of Texas Rio Grande Valley\\
		1201 West University Drive\\
		Edinburg, TX 78539-2999, USA.}

	\email{$^1$meghapandey1071996@gmail.com, $^2$mrinal.roychowdhury@utrgv.edu}


	\date{}
	\maketitle
	\pagestyle{myheadings}\markboth{Megha Pandey and Mrinal K. Roychowdhury}{Constrained quantization for probability distributions}
	
\begin{abstract}
In this work, we extend the classical framework of quantization for Borel probability measures defined on normed spaces $\mathbb{R}^k$ by introducing and analyzing the notions of the $n$th constrained quantization error, constrained quantization dimension, and constrained quantization coefficient. These concepts generalize the well-established $n$th quantization error, quantization dimension, and quantization coefficient, traditionally considered in the unconstrained setting, and thereby broaden the scope of quantization theory. A key distinction between the unconstrained and constrained frameworks lies in the structural properties of optimal quantizers. In the unconstrained setting, if the support of $P$ contains at least $n$ elements, then the elements of an optimal set of $n$-points coincide with the conditional expectations over their respective Voronoi regions; this characterization does not, in general, persist under constraints. Moreover, it is known that if the support of $P$ contains at least $n$ elements, then any optimal set of $n$-points in the unconstrained case consists of exactly $n$ distinct elements. This property, however, may fail to hold in the constrained context.
Further differences emerge in asymptotic behaviors. For absolutely continuous probability measures, the unconstrained quantization dimension is known to exist and equals the Euclidean dimension of the underlying space. In contrast, we show that this equivalence does not necessarily extend to the constrained setting. Additionally, while the unconstrained quantization coefficient exists and assumes a unique, finite, and positive value for absolutely continuous measures, we establish that the constrained quantization coefficient can exhibit significant variability and may attain any nonnegative value, depending critically on the specific nature of the constraint applied to the quantization process.
\end{abstract}

\maketitle

\section{Introduction}
The most common form of quantization is rounding-off. Its purpose is to reduce the cardinality of the representation space, in particular, when the input data is real-valued. It has broad applications in communications, information theory, signal processing, and data compression (see \cite{BW, GG, GL, GN, GL1,  Z1, Z2, P}).
For $k\in \mathbb{N}$, where $\mathbb{N}$ is the set of natural numbers, let $d$ be a metric induced by a norm $\|\cdot\|$ on $\D R^k$.  
Let $P$ be a Borel probability measure on $\D R^k$  and $r \in (0, \infty)$. Let $S$ be a nonempty closed subset of $\D R^k$ and $\ga\ci S$ be a locally finite (i.e., intersection of $\ga$ with any bounded subset of $\D R^k$ is finite) subset of $\D R^k$. This implies that $\ga$ is countable and closed. Then, the distortion error for $P$, of order $r$, with respect to the set $\ga \ci S$, denoted by $V_r(P; \ga)$, is defined by
\begin{equation*}
	V_r(P; \ga)= \int \mathop{\min}\limits_{a\in\ga} d(x, a)^r dP(x).
\end{equation*}
Then, for $n\in \mathbb{N}$, the \tit {$n$th constrained quantization
	error} for $P$, of order $r$, with respect to the set $S$, is defined by
\begin{equation} \label{Vr} V_{n, r}:=V_{n, r}(P)=\inf \Big\{V_r(P; \ga) : \ga \ci S, ~ 1\leq  \text{card}(\ga) \leq n \Big\},\end{equation}
where $\te{card}(A)$ represents the cardinality of a set $A$. If in the definition of $n$th constrained quantization error, the set $S$, known as \tit{constraint}, is chosen as the set $\D R^k$ itself,
then the \tit{$n$th constrained quantization error} is referred to as the \tit{$n$th unconstrained quantization error}, which traditionally in the literature is referred to as the \tit{$n$th quantization error}. For some recent work in the direction of unconstrained quantization, one can see \cite{GL2, GL3,  DFG, DR,  KNZ, PRRSS,P1, R1, R2, R3, R4, RS}. For the probability measure $P$, we assume that 
\begin{equation} \label{M1eq}
	\int d(x, 0)^r dP(x)<\infty.
\end{equation}
A set $\ga$ for which the infimum in \eqref{Vr} exists and does not contain more than $n$ elements is called a constrained optimal set of $n$-points for $P$ with respect to the constraint $S$. The collection of all optimal sets of $n$-points for $P$ with respect to the constraint $S$ is denoted by $\C C_{n, r}(P; S)$.

\begin{prop}\label{prop0}
	Let the assumption~\eqref{M1eq} be true. Then, $V_{n, r}(P)$ exists and is a decreasing sequence of finite nonnegative numbers.
\end{prop}

\begin{proof} 
Assume that condition \eqref{M1eq} holds, i.e.,
\[
\int d(x,0)^r\, dP(x) < \infty.
\]
We first show that this implies
\[
\int d(x,a)^r\, dP(x) < \infty \quad \text{for every } a\in \mathbb{R}^k.
\]
By the triangle inequality,
\[
d(x,a) \le d(x,0) + d(0,a) \le 2\max\{d(x,0),\, d(0,a)\}.
\]
Raising both sides to the power \(r\) (where \(0<r<\infty\)) yields
\[
d(x,a)^r \le 2^r \max\{d(x,0)^r,\, d(0,a)^r\}
\le 2^r\big(d(x,0)^r + d(0,a)^r\big).
\]
Integrating with respect to \(P\), we obtain
\[
\int d(x,a)^r\, dP(x)
\le 2^r\left(\int d(x,0)^r\, dP(x) + d(0,a)^r\right)
< \infty,
\]
since the first integral is finite by \eqref{M1eq} and \(d(0,a)^r\) is a finite constant. Hence,
\[
\int d(x,a)^r\, dP(x) < \infty \quad \text{for all } a\in\mathbb{R}^k. \tag{$\star$}
\]
Now, let \(\alpha \subseteq S \subseteq \mathbb{R}^k\) be any nonempty finite set, and consider
\[
V_r(P;\alpha) = \int \min_{a\in\alpha} d(x,a)^r\, dP(x).
\]
Since \(\min_{a\in\alpha} d(x,a)^r \ge 0\), we have \(V_r(P;\alpha) \ge 0\). Moreover, for any fixed \(a_0\in\alpha\),
\[
\min_{a\in\alpha} d(x,a)^r \le d(x,a_0)^r,
\]
and therefore, using \((\star)\),
\[
0 \le V_r(P;\alpha)
= \int \min_{a\in\alpha} d(x,a)^r\, dP(x)
\le \int d(x,a_0)^r\, dP(x)
< \infty.
\]
Thus, every admissible finite set \(\alpha\) yields a finite nonnegative value of \(V_r(P;\alpha)\).

Consequently, when we define
\[
V_{n,r}(P)
= \inf\left\{ V_r(P;\alpha) : \alpha \subseteq S,\ 1 \le \mathrm{card}(\alpha) \le n \right\},
\]
the infimum is taken over a nonempty collection of finite nonnegative numbers. Hence,
\[
0 \le V_{n,r}(P) < \infty,
\]
showing that \(V_{n,r}(P)\) exists as a finite nonnegative number.

Finally, the monotonicity follows immediately from the definition. If \(m>n\), then
\[
\{\alpha \subseteq S : 1 \le \mathrm{card}(\alpha) \le n\}
\subseteq
\{\alpha \subseteq S : 1 \le \mathrm{card}(\alpha) \le m\},
\]
so taking infima over larger sets cannot increase the value. Therefore,
\[
V_{m,r}(P) \le V_{n,r}(P),
\]
and thus \(\{V_{n,r}(P)\}_{n\ge1}\) is a decreasing sequence. This completes the proof. 
\end{proof}
 
\begin{remark}
The moment condition~\eqref{M1eq} is being assumed throughout the paper. 
\end{remark}

The following two propositions reflect two important properties of constrained quantization. 

\begin{prop} \label{prop01}
	In constrained quantization for any Borel probability measure $P$, an optimal set of one-point always exists, i.e., $\C C_{1, r}(P; S)$ is nonempty. 
\end{prop} 

\begin{proof}
	Let $0<r<\infty$. Define a function 
	\[\psi_r : S \to \D R^{+}: \psi_r(a)=\int d(x, a)^r\,dP(x).\]
	The function $\psi_r$ is obviously continuous. Then, for every $c\in \D R^{+}$ with 
\[c>\mathop{\inf}\limits_{b\in S} \int d(x, b)^r\,dP(x),\]
 the level sets 
	\[\set{\psi_r\leq c}:=\set{a\in S : \int d(x, a)^r\,dP(x)\leq c},\]
	are closed subsets of $S$. 
	Proceeding in the similar way, as shown in the previous proposition, for any $0<r<\infty$, we have
	\begin{align*}  
		d(0, a)^r \leq 2^r\Big( d(x, a)^r+d(x, 0)^r\Big).
	\end{align*} 
	
	Thus, for $a\in \set{\psi_r\leq c}$, we have 
	
	\begin{align*} 
		d(0, a)^r=\int d(0, a)^r \,dP(x)&\leq 2^r\Big(\int (d(x, a)^r+d(x, 0)^r) \,dP(x)\Big)\leq 2^r\Big(c+E\|X\|^r\Big),
	\end{align*}  
	i.e., \[d(0, a)\leq 2\Big(c+ E\|X\|^r\Big)^{\frac 1r} \te{ yielding } \set{\psi_r\leq c} \ci B\Big(0, 2\Big(c+ E\|X\|^r\Big)^{\frac 1r}\Big),\]
	where $E\|X\|^r=\int d(x, 0)^r \,dP(x)$.
	Hence, the level sets are bounded. As the level sets are both bounded and closed, they are compact. 
	Let us now consider a decreasing sequence $\set{c_n}$ of elements in $\D R^+$ such that 
	\begin{equation} \label{eq090} c_n>\mathop{\inf}\limits_{b\in S} \int d(x, b)^r\,dP(x) \te{ and } c_n\to \mathop{\inf}\limits_{b\in S} \int d(x, b)^r\,dP(x).
	\end{equation} 
	Then,  
	\[\set{a\in S : \int d(x, a)^r\,dP(x)\leq c_{n+1}}\ci \set{a\in S : \int d(x, a)^r\,dP(x)\leq c_n}, \te{i.e.,}\] 
	 \[\set{\psi_r\leq c_{n+1}}\ci \set{\psi_r\leq c_n}.\]
	Also, $\set{\psi_r\leq c_n}\neq \es$   for all $n\in \D N$. If this is not true, then there will exist an element $c_N$ for some $N\in \D N$ such that  
	for all $b\in S$, we have 
	\[c_N<  \int d(x, b)^r\,dP(x) \te{ yielding } c_N\leq  \inf_{b\in S} \int d(x, b)^r\,dP(x),\]
	which contradicts \eqref{eq090}. 
	Thus, we see that the level sets $\set{\psi_r\leq c_n}$
	form a nested sequence of nonempty compact sets, and hence, 
	\[\II_{n=1}^\infty\set{\psi_r\leq c_n}\neq \es.\]
	Let $a \in \mathop{\II}\limits_{n=1}^\infty\set{\psi_r\leq c_n}$. Then, 
	\[\mathop{\inf}\limits_{b\in S} \int d(x, b)^r\,dP(x)\leq \int d(x, a)^r\,dP(x)<c_n \te{ for all } n\in \D N,\]
	which by squeeze theorem implies that 
	\[\mathop{\inf}\limits_{b\in S} \int d(x, b)^r\,dP(x)=\int d(x, a)^r\,dP(x),\]
	i.e., $\set{a}$ forms an optimal set of one-point, i.e., $\set{a} \in \C C_{1, r}(P; S)$, i.e.,  $\C C_{1, r}(P; S)$ is nonempty. 
	Thus, the proof of the proposition is complete. 
\end{proof} 

\begin{defi1} Let $P$ be a Borel probability measure on $\D R^k$, and $U$ be the largest open subset of $\D R^k$ such that $P(U)=0$. Then, $\D R^k\setminus U$ is called the \tit{support of $P$}, and is denoted by $\te{supp}(P)$. 
	For a locally finite set $\ga \ci \D R^k$,  and $a\in \ga$, by $M(a|\ga)$ we denote the set of all elements in $\D R^k$ which are nearest to $a$ among all the elements in $\ga$, i.e.,
	\[M(a|\ga)=\set{x \in \D R^k : d(x, a)=\mathop{\min}\limits_{b \in \ga}d(x, b)}.\]
	$M(a|\ga)$ is called the \tit{Voronoi region} in $\D R^k$ generated by $a\in \ga$.
	The set $\set{M(a|\ga) : a \in \ga}$ is called the \tit{Voronoi diagram} or \tit{Voronoi tessellation} of $\D R^k$ with respect to the set $\ga$. Further, for $\alpha=\{a_1,a_2,\ldots\}\subseteq S$, let us define the sets $A_{a_i|\alpha}$ for $a_i\in \ga$ as follows:
	
	\begin{equation}\label{Megha1000}
		\begin{aligned}
			A_{a_{i}|\alpha}=\left\{\begin{array} {ll}
				M(a_1|\ga) & \te{ if } i=1,\\
				M(a_i|\ga)\setminus \UU_{k<i} M(a_k|\ga) & \te{ if } i\geq 2.
			\end{array}\right.
		\end{aligned}
	\end{equation} 
	The set $\set{A_{a_{i}|\alpha} : a_i\in \ga}$ is called the \tit{Voronoi partition} of $\D R^k$ with respect to the set $\ga$ (and $S$). 
\end{defi1}
\begin{prop} \label{Megha1} 
	Let $P$ be a Borel probability measure on $\D R^k$ and $S$ be a nonempty closed subset of $\D R^k$. Let $\alpha_n$ be an optimal set of $n$-points for $P$ with respect to the constraint $S$. Then, $\alpha_n$ contains exactly $n$ elements if and only if there exists a set $\ga\subseteq S$ containing at least $n$ elements such that 
	\[P(A_{a|\alpha})>0 \te{ for each } a \in \ga.\]  			
\end{prop}
\begin{proof}
	Let us first assume that there exists a set $\ga\subseteq S$ containing at least $n$ elements such that 
	\[P(A_{a|\alpha})>0 \te{ for each } a \in \ga.\]
	If $n=1$, then the proposition is a consequence of Proposition~\ref{prop01}. Let us now prove the proposition for $2\leq n $.  For the sake of contradiction, assume that $\gg:=\set{b_1, b_2, \cdots, b_m}\subseteq S$ is an optimal set of $n$-points for $P$, such that $\te{card}(\gg)=m$ for some positive integer $m<n$, i.e., 
	\[V_{n, r}(P) =\int\min_{a\in \gg} d(x, a)^r dP(x).\]
	If $m\geq 2$, then we have 
	\[V_{n, r}(P) =\sum_{j=1}^m \int_{A_{b_{j}|\gg}}  d(x, b_j)^r dP(x).\]
	Since $\te{card}(\ga)>\te{card}(\gg)$, let $c \in \ga$ be such that $c \not \in \set{b_j : 1\leq j\leq m}$. Then, $c \in A_{b_{\ell}|\gg}$ for some $1\leq \ell\leq m$. Consider the set $\gg \uu \set{c}$, we can  write 
	\[\gg \uu \set{c}=\set{b_1, b_2, \cdots, b_m, b_{m+1}}, \te{ where } b_{m+1}=c.\]
	Notice that $A_{c|\gg\uu \set{c}}$ intersects some of the Voronoi partitions of $A_{b_{j}|\gg }$ for $1\leq j\leq m$. If $A_{c|\gg\uu \set{c}}$ does not intersect $A_{b_{j}|\gg}$ for some $1\leq j\leq m$, then 
	\begin{align} \label{Meg110} \int_{A_{b_{j}|\gg}}d(x, b_j)^r dP(x)=\int_{A_{b_{j}|\gg\uu\set{c}}} d(x, b_j)^r dP(x).
	\end{align} 
	On the other hand, if $A_{c|\gg\uu \set{c}}$ intersects $A_{b_{j}|\gg}$ for some $1\leq j\leq m$, then 
	\begin{equation}\label{Meg111} 
	\begin{aligned} 
		&\int_{A_{b_j|\gg}}d(x, b_j)^r dP(x)\\
=&\int_{A_{b_j|\gg}\ii A_{c|\gg\uu \set{c}}}d(x, b_j)^r dP(x)+\int_{A_{b_j|\gg}\ii A_{c|\gg\uu \set{c}}^c}d(x, b_j)^r dP(x)\\
		>&\int_{A_{b_j|\gg}\ii A_{c|\gg\uu \set{c}}}d(x, c)^r dP(x)+\int_{A_{b_j|\gg}\ii A_{c|\gg\uu \set{c}}^c}d(x, b_j)^r dP(x). 
	\end{aligned}
\end{equation}
	For $1\leq t<m$, let $A_{c|\gg\uu \set{c}}$ intersects $A_{b_{\ell_i}|\gg}$, where $b_{{\ell_i}}\in \set{b_j : 1\leq j\leq m} $ for $i\in \set{1, 2, \cdots t}$. Then, 
	\[A_{c|\gg\uu\set{c}}=\UU_{i=1}^t A_{b_{{\ell_i}}|\gg}\ii A_{c|\gg\uu \set{c}}.\]
	Hence, by the expressions in \eqref{Meg110} and \eqref{Meg111}, we have 
	\begin{align*}
		V_{n, r}(P)=&\int\min_{a\in \gg} d(x, a)^r dP(x)\\
		>&  \sum_{j\in (\set{1, 2, \cdots, m}\setminus\set{\ell_1, \ell_2, \cdots, \ell_t})}\int_{A_{b_j|\gg\uu\set{c}}} d(x, b_j)^r dP(x)\\
		&+\sum_{i=1}^t \Big(\int_{A_{b_{{\ell_i}}|\gg}\ii A_{c|\gg\uu \set{c}}}d(x, c)^r dP(x)+\int_{A_{b_{{\ell_i}}|\gg}\ii A_{c|\gg\uu \set{c}}^c}d(x, b_{{\ell_i}})^r dP(x)\Big)\\
		=&  \sum_{j\in (\set{1, 2, \cdots, m}\setminus\set{\ell_1, \ell_2, \cdots, \ell_t})}\int_{A_{b_j|\gg\uu\set{c}}} d(x, b_j)^r dP(x)+  \int_{A_{c|\gg\uu \set{c}}}d(x, c)^r dP(x)\\
		&  +\sum_{i=1}^t\int_{A_{b_{{\ell_i}}|\gg}\ii A_{c|\gg\uu \set{c}}^c}d(x, b_{{\ell_i}})^r dP(x)\\
		=&  \sum_{j\in (\set{1, 2, \cdots, m}\setminus\set{\ell_1, \ell_2, \cdots, \ell_t})}\int_{A_{b_j|\gg\uu\set{c}}} d(x, b_j)^r dP(x)+  \int_{A_{c|\gg\uu \set{c}}}d(x, c)^r dP(x)\\
		&+\sum_{i=1}^t\int_{A_{b_{{\ell_i}}|\gg\uu \set{c}}}d(x, b_{{\ell_i}})^r dP(x)\\
		=&\int\min_{a\in \gg\uu \set{c}} d(x, a)^r dP(x)\geq V_{n, r}(P),
	\end{align*} 
	where the last inequality is true since $\gg\uu \set{c}$ contains no more than $n$ elements. Thus, we see that a contradiction arises. Hence, an optimal set of $n$-points contains exactly $n$ elements. 
	
	Next, assume that $\ga_n$ is an optimal set of $n$-points for $P$ with respect to the constraint $S$ such that $\ga_n$ contains exactly $n$ elements. We need to show that
	\[P(A_{a|\alpha_n})>0 \te{ for each } a \in \ga_n.\]  
	For the sake of contradiction, let $\gg$ be the nonempty maximal subset of $\ga_n$ such that $P(A_{b|\ga_n})=0 \te{ for each } b \in \gg$
	and $P(A_{b|\ga_n})>0 \te{ for each } b \in \ga_n\setminus \gg$. Since optimal set of one-point always exists, hence $\gg\neq \ga_n$. Let $\te{card}(\gg)=m$, and then $\te{card}(\ga_n\setminus \gg)=n-m$. Thus, we see that 
	\begin{align*}
		V_{n, r}(P) &=\int\min_{a\in \ga_n} d(x, a)^r dP(x)
		=\sum_{a\in \ga_n} \int_{A_{a|\ga_n}} d(x, a)^r dP(x)\\
		=&\sum_{a\in \ga_n\setminus\gg} \int_{A_{a|\ga_n\setminus \gg}} d(x, b)^r dP(x),
	\end{align*} 
	which implies that $\ga_n\setminus \gg$ is an optimal set of $n$-points such that $\te{card }(\ga_n\setminus \gg)=n-m<n$, which contradicts our assumption. Thus, we see that if $\ga_n$ is an optimal set of $n$-points containing exactly $n$ elements, then
	\[P(A_{a|\ga_n})>0 \te{ for each } a \in \ga_n.\]
	Thus, the proof of the proposition is complete.			
\end{proof}

The following corollaries are direct consequences of Proposition~\ref{Megha1}. 

\begin{cor}
	Let $P$ be a Borel probability measure on $\D R^k$ and $S$ be a nonempty closed subset of $\D R^k$ such that there exists a set $\ga\subseteq S$ containing at least $N$ elements for some positive integer $N$ such that
	\[P(A_{a|\alpha})>0 \te{ for each } a \in \ga.\]
	Then, the sequence $\set{V_{n, r}(P)}_{n=1}^N$ is strictly decreasing, i.e., $V_{n-1,r}(P)>V_{n,r}(P)$ for all $2\leq n\leq N$, 
	where $V_{n, r}(P)$ represents the $n$th constrained quantization error with respect to the constraint $S$.  
\end{cor}   
\begin{cor}
		Let $P$ be a Borel probability measure on $\D R^k$ and $S$ be a nonempty closed subset of $\D R^k$. Let $\ga$ be an optimal set of $n$-points containing exactly $n$ elements for $P$ with respect to the constraint $S$ and $a\in \ga$. Then, $P(M(a|\ga))>0$.
\end{cor}

Although the following proposition has already been established in \cite{GL}, we provide an alternative proof based on Proposition~\ref{Megha1}. 

\begin{prop} \label{Megha2} 
	Let $P$ be a Borel probability measure on $\D R^k$, if $S=\D R^k$, i.e., when there is no constraint, then an optimal set of $n$-points for $P$ contains exactly $n$ elements if and only if the $\te{supp}(P)$ contains al least $n$ elements. 			
\end{prop}
\begin{proof}
	If an optimal set of $n$-points contains exactly $n$ elements, then it is easy to observe that $\te{supp}(P)\geq n$. Next, assume that $\te{supp}(P)\geq n$. We need to prove that an optimal set of $n$-points for $P$ contains exactly $n$ elements. In view of Proposition \ref{Megha1}, it is sufficient to prove that there exists a subset $\alpha\ci \D R^k$ containing at least $n$ elements such that $P(A_{a|\alpha})>0 \te{ for all } a\in \alpha$. For the sake of contradiction, let $\gamma=\{a_1,a_2,\ldots a_m\}\ci \D R^k$ be a subset of maximal element $m<n$ such that $P(A_{a_i|\gamma})>0$ for all $a_i\in \gamma$. Since $\te{supp}(P)\geq n$, there exists $b\in \te{supp}(P)$ with $b\notin \gamma$ such that $b\in A_{a_{\ell}|\gamma}$ for some $1\leq \ell\leq m$. Then, proceeding analogously as the proof of Proposition \ref{Megha1}, we can prove that $\gamma\cup\{b\}$ is a set of $m+1$ elements such that $P(A_{a|\gamma\cup\{b\}})>0$ for all $a\in \gamma\cup \{b\}$, which gives a contradiction. Thus, we deduce that an optimal set of $n$-points contains exactly $n$ elements. Hence, the proof of the  proposition is complete. 
\end{proof} 

The following proposition is a standard result in quantization theory (see \cite{GL}).  However, for the sake of completeness, we provide a proof here.

\begin{prop}\label{optimalnmeans}
	Let $P$ be a Borel probability measure on $\D R^k$, if $S=\D R^k$, i.e., when there is no constraint, then the elements in an optimal set of $n$-points are the conditional expectations in their own Voronoi regions provided that the $\te{supp}(P)$ contains at least $n$ elements. 
\end{prop}
\begin{proof}
	Let $\alpha_n$ be an optimal set of $n$-points for $P$ and $\set{A_{a|\alpha_n} : a \in \ga_n}$ be the  Voronoi partition of $\D R^k$ with respect to the set $\ga_n$. Let $V_{n,r}(P)$ be the corresponding $n$th quantization error. Then, 
	\begin{align*}
		V_{n,r}(P)= \int \min_{a\in \alpha_n}d(x,a)^r\ dP(x)
		=\sum_{a\in \alpha_n}\int_{A_{a|\alpha_n}}d(x,a)^r\ dP(x). 
	\end{align*}
	Notice that $V_{n,r}(P)$ will be minimum if the function \[\mathsf{F}_{r}(a):=\int_{A_{a|\alpha_n}} d(x,a)^r \ dP(x)\]
	is minimum for each $a\in \ga_n$.  
	Notice that the value of $a$, where the function $F_r(a)$ is minimum, does not depend on $r$ because for any $u, v, w, z\in\D R^k$ and $r>0$, we have  
	\begin{align*}
		d(u, v)\leq d(w, z) \te{ if and only if } d(u, v)^r\leq d(w, z)^r.
	\end{align*}
	Hence, for simplicity, we calculate the value of $a$, where $F_r(a)$ is minimum, for $r=2$. For this, first we compute gradient $\nabla \mathsf{F}_2(a)$. We have
	\begin{align*}
		\nabla \mathsf{F}_2(a) = \frac{d}{da} \int_{A_{a|\alpha_n}} d(x,a)^2 \, dP(x)
		=& \frac{d}{da} \int_{A_{a|\alpha_n}} \|x - a\|^2 \, dP(x)\\
	    =& \int_{A_{a|\alpha_n}} \nabla_a \|x - a\|^2 \, dP(x).
	\end{align*}
	Since,
	\[\nabla_a \|x - a\|^2 = \nabla_a \left[ (x - a)^T (x - a) \right] = -2(x - a),\]
	where $x^T$ gives the transpose of $x\in \mathbb R^k$, we have 
	\[\nabla \mathsf{F}_2(a) = \int_{A_{a|\alpha_n}} -2(x - a) \, dP(x) = -2 \left( \int_{A_{a|\alpha_n}} x \, dP(x) - a \cdot P(A_{a|\alpha_n}) \right).\]
	Set this to zero to find the minimizer:
	\[-2 \left( \int_{A_{a|\alpha_n}} x \, dP(x) - a \cdot P(A_{a|\alpha_n}) \right) = 0.\]
	Solve for \( a \):
	\[a = \frac{1}{P(A_{a|\alpha_n})} \int_{A_{a|\alpha_n}} x \, dP(x) = E(X \mid X \in A_{a|\alpha_n}).\]
	Hence, the elements in an optimal set of $n$-points are the conditional expectations in their own Voronoi regions. Thus, the proof of the proposition is complete. 
\end{proof}

\begin{remark}
Let $P$ be a Borel probability measure on $\D R^k$ and $S$ be a nonempty closed subset of $\D R^k$. Let $\ga$ be an optimal set of $n$-points for $P$ with respect to the constraint $S$ and $a\in \ga$. If the probability measure $P$ is absolutely continuous with respect to the Lebesgue measure on $\mathbb{R}^k$, then it is easy to observe that $ P(\partial M(a|\ga))=0$, where $\partial M(a|\ga)$ represents the boundary of the Voronoi region $M(a|\ga)$. However, when $P$ is not absolutely continuous, this property may fail, even in constrained quantization, as shown in the following example.
\end{remark}

\begin{exam} 
Let $P$ be a Borel probability measure on $\D R$ which is discrete and uniform on its support $\set{1, 2, 3, 4, 5, 6, 7}$. Let us take the constraint as 
\[S=\set{(2.5, 1), (5.5, 1)}\uu \set{(x, 1) : x\in\D R \te{ and } x\geq 12}.\]
Let $\ga_n$ denote a constrained optimal set of $n$-points with respect to the constraint $S$ for the probability distribution $P$. Then,  
 \begin{align*}
 \ga_1&=\set{(2.5, 1)}, \te{ or } \set{(5.5, 1)},  \te{ and } \ga_2=\set{(2.5, 1), (5.5, 1)}.
 \end{align*} 
Write $a_1=(2.5, 1)$ and $a_2=(5.5, 1)$. Then, notice that $P(\partial M(a_1|\ga_2))=P(\partial M(a_2|\ga_2))=P(\set{4})=\frac 17 \neq 0$, exactly as intended. 
\end{exam}

In view of Proposition \ref{optimalnmeans}, in the case of unconstrained quantization, the elements in an optimal set are the conditional expectations in their own Voronoi regions. However, as will be demonstrated in later sections, this characterization does not generally hold in the context of constrained quantization. Because of that, in the case of constrained quantization,  a set $\ga$ for which the infimum in \eqref{Vr} exists and contains no more than $n$ elements is called an \tit{optimal set of $n$-points} instead of calling it as an \tit{optimal set of $n$-means}. Elements of an optimal set are called \tit{optimal elements}. In unconstrained quantization, as described in \cite {GL}, if the support of $P$ contains at least $n$ elements, then an optimal set of $n$-means always contains exactly $n$ elements. However, this property does not carry over to constrained quantization. In particular, while an optimal set of one-point containing exactly one element in the constrained setting always exists, an optimal set of $n$-points containing exactly $n$ elements for $n\geq 2$ may not exists, even if the support of $P$ has at least $n$ elements. Notice that unconstrained quantization, as described in \cite{GL}, is a special case of constrained quantization. Nonetheless, there are some properties that hold in the unconstrained case, and do not extend to the constrained setting.

This paper deals with $r=2$ and $k=2$, and the metric on $\D R^2$ as the Euclidean metric induced by the Euclidean norm $\|\cdot\|$. Instead of writing $V_r(P; \ga)$ and $V_{n, r}:=V_{n, r}(P)$ we will write them as $V(P; \ga)$ and $V_n:=V_{n}(P)$, i.e., $r$ is omitted in the subscript as $r=2$ throughout the paper.

\subsection{Motivation and Work Done}
There are several research works in the literature on unconstrained quantization (see, for instance, 
\cite{GG, GL, GN, GL1, Z1, Z2, P, P1, R1, R2, R3, R4, RS}), and it has proved effective in solving problems 
for various probability distributions, such as uniform and self-similar distributions. However, in many 
real-life situations, quantization is subject to spatial or geometric restrictions.

A key motivating example arises in radiation therapy planning. In such applications, radiation beams 
or sampling points cannot be placed arbitrarily in space; instead, they must avoid sensitive or healthy 
tissue. Mathematically, this restriction can be modeled by introducing a constraint set $S \subset \mathbb{R}^k$ 
representing the \emph{admissible region}, where quantizer points are allowed to lie. The complement 
$\mathbb{R}^k \setminus S$ can then be interpreted as a \emph{forbidden region}, corresponding to organs 
or tissue that must not be directly exposed to radiation. Consequently, the optimization problem becomes 
one of minimizing quantization error subject to the requirement that all optimal elements lie in $S$.

This naturally leads to the notion of \emph{constrained quantization}, where the locations of optimal 
quantizer points are restricted by geometric, physical, or safety considerations. In this paper, we introduce 
and analyze constrained quantization for probability distributions by defining the constrained 
quantization error, constrained quantization dimension, and constrained quantization coefficient. 
We further compute optimal sets of $n$-points for several distributions under different constraint 
geometries, illustrating how constraints significantly alter both optimal configurations and asymptotic 
behavior.

To make the theoretical consequences of constrained quantization concrete, we focus on 
geometrically simple yet structurally informative model cases, such as uniform distributions 
supported on a line segment, a circle, and a chord, together with natural constraint sets defined 
by lines or circles. These examples permit explicit analytical computation of optimal configurations 
while highlighting phenomena unique to the constrained setting, including geometric mismatch 
between the support of the distribution and the admissible quantizer locations, potential 
degeneracy of optimal sets, and strong sensitivity of quantization behavior to constraint geometry.

\subsection{Relation to constrained optimization, facility location, and CVT literature.}
 The constrained quantization problem considered here can be viewed as a continuous
facility-location or clustering problem in which one seeks a finite set of ``facilities''
$\alpha \subset S$ minimizing the expected squared distance
$\int \min_{a\in \alpha}\|x-a\|^2\,dP(x)$, subject to the geometric feasibility constraint
$\alpha \subset S$. This places our framework in close relation to classical continuous
$k$-means and $k$-median type objectives, as well as facility-location models in operations
research and spatial optimization, where admissible facility positions are restricted by
geometric, feasibility, or safety considerations \cite{GN, MacQueen1967,Lloyd1982,LoveMorrisWesolowsky1988,  DreznerHamacher2004}.

There is also a strong conceptual connection with centroidal Voronoi tessellations (CVTs)
and their constrained variants (CCVTs), in which generators are restricted to lie on a
prescribed region (often a curve or surface), and one seeks minimizers of Voronoi-based
energy functionals \cite{DFG, DuWang2005, DuGunzburgerJu2003}.
In contrast to the standard CVT/CCVT setting, our framework allows the probability measure
$P$ to be supported on a set that may be geometrically different from the constraint set $S$
(e.g., a line segment, chord, or circle for the support versus a distinct curve or surface
for the admissible quantizer locations). This ``mismatched geometry'' between
$\operatorname{supp}(P)$ and $S$ is a key source of the phenomena we highlight:
(i) optimal sets in constrained quantization need not satisfy the classical centroid
(conditional expectation) characterization from the unconstrained case,
(ii) optimal sets of $n$-points may fail to contain $n$ distinct points (and may not exist
for certain $n$), and (iii) the constrained quantization dimension and coefficient can depend
strongly on the geometry of $S$ even when $\operatorname{supp}(P)$ is fixed
\cite{GL, Z1}.

These features position constrained quantization as a geometric extension of classical
quantization theory that complements, but is not encompassed by, the standard CVT/CCVT
literature. Beyond computing optimal configurations for concrete constrained geometries,
our work also develops asymptotic invariants (the constrained quantization dimension and
coefficient) that quantify how geometric constraints fundamentally alter quantization rates.
\subsection{Delineation}
The organization of the paper is as follows. Section~\ref{sec000} presents the preliminaries that will be used throughout the paper. Section~\ref{sec1} studies constrained quantization for the uniform distribution supported on a closed interval $[a,b]$, where the optimal elements are restricted to lie on another line segment. Section~\ref{sec2} investigates constrained quantization for the uniform distribution supported on a circle, with optimal elements constrained to lie on another circle. In Section~\ref{sec3}, we analyze the case where the uniform distribution is supported on a chord of a circle, while the optimal elements lie on the circle itself. Section~\ref{sec5} introduces the notions of constrained quantization dimension and constrained quantization coefficient, and through several examples highlights the differences between constrained and unconstrained quantization dimensions and coefficients. Finally, Section~\ref{sec6} outlines future directions and open problems arising from the work presented in this paper.

\section{Preliminaries} \label{sec000}
For any two elements $(a, b)$ and $(c, d)$ in $\D R^2$, we write 
\[\rho((a, b), (c, d)):=(a-c)^2 +(b-d)^2,\] which gives the squared Euclidean distance between the two elements $(a, b)$ and $(c, d)$.
Two elements $p$ and $q$ in an optimal set of $n$-points are called \tit{adjacent elements} if they have a common boundary in their own Voronoi regions. Let $e$ be an element on the common boundary of the Voronoi regions of two adjacent elements $p$ and $q$ in an optimal set of $n$-points. Since the common boundary of the Voronoi regions of any two elements is the perpendicular bisector of the line segment joining the elements, we have
\[\rho(p, e)-\rho(q, e)=0. \]
We call such an equation a \tit{canonical equation}. 
\begin{fact} \label{fact}
	Let $P$ be a Borel probability measure on $\D R^2$ with support a curve $C$ given by the parametric representations
	\[x=F(t) \te{ and } y=G(t), \te{ where }  a\leq t\leq b.\]
	Let us fix a point on the curve $C$. Let $s$ be the distance of a point on the curve tracing along the curve starting from the fixed point. Then, 
	\begin{equation} \label{eq000} ds=\sqrt{(dx)^2+(dy)^2}=\sqrt{(F'(t))^2+(G'(t))^2} |dt|,
	\end{equation}
	where $d$ stands for differential.  
	Notice that $|dt|=dt$ if $t$ increases, and $|dt|=-dt$ if $t$ decreases. Then, 
	\[dP(s)=P(ds)=f(x, y) ds,\]
	where $f(x, y)$ is the probability density function for the probability measure $P$, i.e., $f(x, y)$ is a real-valued function on $\D R^2$ with the following properties:
	$f(x, y)\geq 0$ for all $(x, y)\in \D R^2$, and
	\begin{align*}
	\int_{\D R^2}f(x, y)\, dA(x, y)=&\int_{\D R^2\setminus C}f(x, y)\, dA(x, y)+\int_{C}f(x, y)\, dA(x, y)\\
	=&\int_{C}f(x, y)\, dA(x, y)=1,
	\end{align*}
	where for any $(x, y) \in \D R^2$ by $dA(x, y)$, it is meant the infinitesimal area $dxdy$, and if $(x, y) \in C$, then by $dA(x,y)$, it is meant the infinitesimal length $ds$ given by \eqref{eq000}. 
\end{fact} 

\section{Constrained quantization when the support lies on a line segment and the optimal elements lie on another line segment} \label{sec1}
Let $a, b\in \D R$ with $a<b$, and $c, m\in \D R$. Let $L$ be a line given by $y=mx+c$, the parametric representation of which is
\[L:=\set{(x, mx+c) : x\in \D R}.\]
Let $P$ be a Borel probability measure on $\D R^2$ such that $P$ is uniform on its support $\set{(x, y) \in \D R^2 : a\leq x\leq b \te{ and } y=0}$.
Then, the probability density function $f$ for $P$ is given by
\[f(x, y)=\left\{\begin{array}{ccc}
	\frac 1 {b-a} & \te{ if } a\leq x\leq b \te{ and } y=0,\\
	0  & \te{ otherwise}.
\end{array}\right.
\]
Recall Fact~\ref{fact}. Here we have  $dP(s)=P(ds)=P(dx)=dP(x)=f(x, 0)dx$. In this section, we determine the optimal sets of $n$-points and the $n$th constrained quantization errors for the probability measure $P$ for all positive integers $n$ so that the elements in the optimal sets lie on the line $L$ between the two elements $(d, md+c)$ and $(e, me+c)$, where $d, e\in \D R$ with $d<e$.
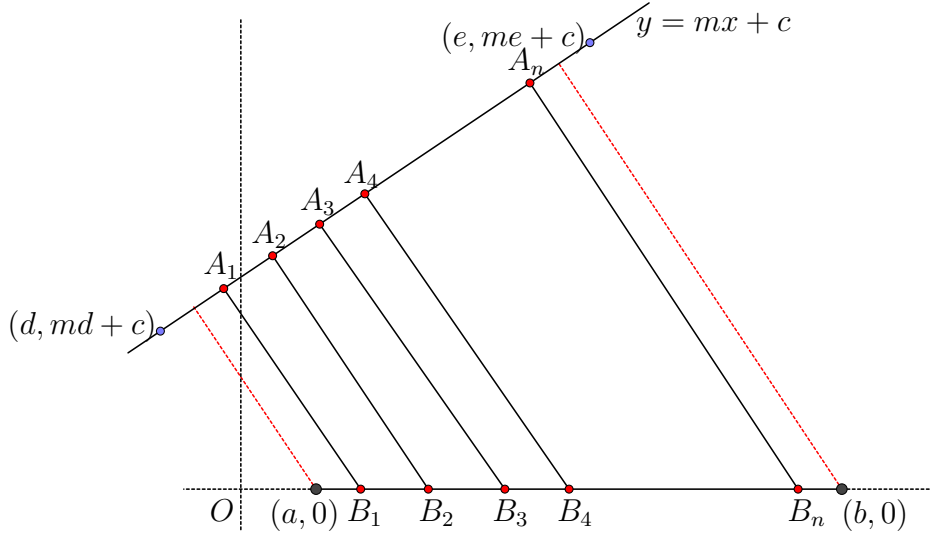
\begin{figure}
	\begin{tikzpicture}[line cap=round,line join=round,>=triangle 45,x=1.0cm,y=1.0cm]
		\clip(-3.626547028728845,-0.6796800380732857) rectangle (10.543925134766551,10.012054583893859);
		\draw [line width=0.59 pt] (1.,0.)-- (8.,0.);
		\draw [line width=0.59 pt,dash pattern=on 1pt off 1pt,color=ffqqqq] (-0.6250355564097686,2.4237324620411176)-- (1.,0.);
		\draw [line width=0.59 pt,dash pattern=on 1pt off 1pt,color=ffqqqq] (4.2336997885835,5.652638862194866)-- (8.,-0.015375744762637034);
		\draw [line width=0.59 pt] (0.4051393426869056,3.1310167211224087)-- (2.4962406304055276,-0.005635210455521756);
		\draw (-0.642493872582671,3.2984031190796207) node[anchor=north west] {$A_1$};
		\draw (-0.0015294106879773675,3.7016408116196375) node[anchor=north west] {$A_2$};
		\draw (0.6427875126987127,4.148770581074987) node[anchor=north west] {$A_3$};
		\draw (1.2626727437467337,0.011135791621818) node[anchor=north west] {$B_1$};
		\draw (0.2361973586666987,0.0111232942533151627) node[anchor=north west] {$(a, 0)$};
		\draw (7.84108362391646,0.01113454527150085072) node[anchor=north west] {$(b, 0)$};
		\draw (2.218011513504107,0.01111540056560775146) node[anchor=north west] {$B_2$};
		\draw (3.2016740525154817,0.0111540056560775146) node[anchor=north west] {$B_3$};
		\draw (3.4016740525154817,6.003397581981053) node[anchor=north west] {$A_n$};
		\draw (7.122546900597628,0.011140056560775146) node[anchor=north west] {$B_n$};
		\draw (-.5559037185506569, 0.0122150224684905) node[anchor=north west] {$O$};
		\draw [line width=0.59 pt] (-1.4995662575620319,1.813121733798239)-- (5.4268406688448865,6.468311505267068);
		\draw [line width=0.59 pt,dash pattern=on 1pt off 1pt] (0.,-0.5386350365631762)-- (0.,6.242800582081727);
		\draw (5.1091926745558871,6.487771889843732) node[anchor=north west] {$y=mx+c$};
		\draw [line width=0.59 pt] (3.843439288855313,5.404118484762449)-- (7.408115208258959,0.);
		\draw [line width=0.59 pt] (1.0477018762287615,3.5251228748808865)-- (3.5099978217694834,0.);
		\draw [line width=0.59 pt] (1.648502605186379,3.928916853180308)-- (4.363717745256848,0.);
		\draw [line width=0.59 pt,dash pattern=on 1pt off 1pt] (1.,0.)-- (-0.7586017956673382,0.);
		\draw [line width=0.59 pt,dash pattern=on 1pt off 1pt] (7.98978307665631,0.)-- (9.324958055334362,0.);
		\draw (1.2193202822547374, 4.501468658191668) node[anchor=north west] {$A_4$};
		\draw (4.0587464374948425,0.0111568188682351291) node[anchor=north west] {$B_4$};
		\draw [line width=0.59 pt] (-0.2277513914307765,2.6678996321980817)-- (1.5931549746940805,0.);
		\draw (-3.2454446427427414,2.531273349624271) node[anchor=north west] {$(d, md+c)$};
		\draw (2.513889898684816,6.353282813029073) node[anchor=north west] {$(e, me+c)$};
		\begin{scriptsize}
			\draw [fill=ffqqqq] (0.42252149052143523,3.104943499370614) circle (1.5pt);
			\draw [fill=ffqqqq] (3.8434392888553126,5.404118484762449) circle (1.5pt);
			\draw [fill=ffqqqq] (7.408115208258959,0.) circle (1.5pt);
			\draw [fill=ffqqqq] (2.4924838234351796,0.) circle (1.5pt);
			\draw [fill=ffqqqq] (1.0477018762287618,3.525122874880887) circle (1.5pt);
			\draw [fill=ffqqqq] (3.509997821769484,0.) circle (1.5pt);
			\draw [fill=ffqqqq] (1.6485026051863785,3.928916853180307) circle (1.5pt);
			\draw [fill=ffqqqq] (4.363717745256848,0.) circle (1.5pt);
			\draw [fill=uuuuuu] (1.,0.) circle (2.0pt);
			\draw [fill=uuuuuu] (7.98978307665631,0.) circle (2.0pt);
			\draw [fill=ffqqqq] (-0.2277513914307767,2.6678996321980817) circle (1.5pt);
			\draw [fill=ffqqqq] (1.5931549746940807,0.) circle (1.5pt);
			\draw [fill=xdxdff] (-1.0695639663044743,2.102123273736458) circle (1.5pt);
			\draw [fill=xdxdff] (4.640868525937184,5.9400651115453815) circle (1.5pt);
		\end{scriptsize}
	\end{tikzpicture}
	\caption{Support of the probability distribution $P$ is the closed interval joining the points $(a, 0)$ and $(b, 0)$; $A_i (a_i, ma_i+c)$ are the elements in an optimal set of $n$-points lying on the line $y=mx+c$ between the two points $(d, md+c)$ and $(e, me+c)$; $B_i ((m^2+1) a_i+mc, 0)$ are the points where the perpendiculars through $A_i$ on the line $y=mx+c$ intersect the support of $P$.} \label{Fig0}
\end{figure}
Let us now give the following theorem.
\begin{theorem}\label{sec2Theorem1}
	Let $P$ be a Borel probability measure on $\D R^2$ such that $P$ is uniform on its support $\set{(x, y) \in \D R^2 : a\leq x\leq b \te{ and } y=0}$. For $n\in \D N$ with $n\geq 2$, let $\ga_n:=\{(a_i,m a_i+c): 1\leq i \leq n\}$ be an optimal set of $n$-points for $P$ so that the elements in the optimal sets lie on the line $L$ between the two elements $(d, md+c)$ and $(e, me+c)$ (see Figure~\ref{Fig0}), where $d, e\in \D R$ with $d<e$. Assume that
	\[\max\set{a, (m^2+1)d+mc}=a \te{ and } \min\set{b, (m^2+1) e+mc}=b.\]
	Then, $a_i=\frac {2i-1}{2n(1+m^2)}(b-a)+\frac{a-cm}{1+m^2}$ for $1\leq i\leq n$ with constrained quantization error
	\[V_2=\frac{a^2 \left(16 m^2+1\right)+2 a b \left(8 m^2-1\right)+48 a c m+b^2 \left(16 m^2+1\right)+48 b c m+48 c^2}{48 \left(m^2+1\right)}\]
	and for $n\geq 3$, we have
	\begin{align*}
		V_n=&\frac 1{12 \left(m^2+1\right) n^3}\Big(-48 (a - b)^2 m^2 + (a - b) (a - b + 72 c m + 8 (11 a - 2 b) m^2) n \\
		&\qquad \qquad -
		12 (a - b) m (5 c + (4 a + b) m) n^2 + 12 (c + a m)^2 n^3\Big).
	\end{align*} 
\end{theorem}

\begin{proof}
	For $n\geq 2$, let $\ga_n:=\set{(a_i, m a_i+c) :  1\leq i\leq n}$ be an optimal set of $n$-points on $L$ such that $d\leq a_1<a_2<\cdots<a_{n-1}<a_n\leq e$. Notice that the boundary of the Voronoi region of the element $(a_1, ma_1+c)$ intersects the support of $P$ at the elements $(a, 0)$ and $((m^2+1)\frac {(a_1+a_2)}{2}+mc, 0)$, the boundary of the Voronoi region of $(a_n, ma_n+c)$ intersects the support of $P$ at the elements
	$((m^2+1)\frac {(a_{n-1}+a_n)}{2}+mc, 0)$ and $(b,0)$. On the other hand, the boundaries of the Voronoi regions of $(a_i,ma_i+c)$ for $2\leq i \leq n-1$ intersect the support of $P$ at the elements $((m^2+1)\frac {(a_{i-1}+a_i)}{2}+mc, 0)$ and $((m^2+1)\frac {(a_{i}+a_{i+1})}{2}+mc, 0)$. Since the Voronoi regions of the elements in an optimal set must have positive probability, we  have
	\begin{align*} \label{eq123} \max\set{a, (m^2+1) d+mc} &\leq(m^2+1) a_1+mc<(m^2+1) a_2+mc\\
		&<\cdots<(m^2+1) a_n+mc\leq \min\set{b, (m^2+1) e+mc}.
	\end{align*}  
	Let us consider the following two cases:
	
	\setword{\tit{Case 1}}{Word:Case1}: $n=2$.
	
	In this case, the distortion error due to the set $\ga_2$ is given by
	\begin{align*}
		V(P; \ga_2)
		=&\int_{\D R} \min_{a\in \ga_2}\|(x, 0)-a\|^2dP(x)\\
		=&\frac{1}{b-a}\Big(\int_a^{(m^2+1)\frac {(a_{1}+a_2)}{2}+mc} \rho((x, 0), (a_1, ma_1+c)) \, dx \\
		&\qquad  + \int_{(m^2+1)\frac {(a_{1}+a_2)}{2}+mc}^b \rho((x, 0), (a_2, ma_2+c)) \, dx\Big).
	\end{align*}
	Notice that $V(P; \ga_2)$ is not always differentiable with respect to $a_1$ and $a_2$. By the hypothesis, we have
	\[\max\set{a, (m^2+1)d+mc}=a \te{ and } \min\set{b, (m^2+1) e+mc}=b.\]
	This guarantees that $V(P; \ga_2)$ is differentiable with respect to $a_1$ and $a_2$. \\
	Since $\frac{\pa}{\pa a_1} V(P; \ga_2)=0$ and $\frac{\pa}{\pa a_2} V(P; \ga_2)=0$, we deduce that
	\[-3 a_1 m^2+a_2 m^2+2 a-3 a_1+a_2-2 c m=0, \te{ and }\]
	  \[a_1 m^2-3 a_2 m^2+a_1-3 a_2+2 b-2 c m=0\]
	implying
	\begin{equation*} \label{eq2010}
		a_1=\frac 1{4(1+m^2)}(b-a)+\frac{a-cm}{1+m^2}~ \te{ and } ~a_2=\frac 3{4(1+m^2)}(b-a)+\frac{a-cm}{1+m^2}
	\end{equation*}
	with quantization error
	\[V_2=\frac{a^2 \left(16 m^2+1\right)+2 a b \left(8 m^2-1\right)+48 a c m+b^2 \left(16 m^2+1\right)+48 b c m+48 c^2}{48 \left(m^2+1\right)}.\]
	\tit{Case~2: $n\geq 3$.}
	
	In this case, the distortion error due to the set $\ga_n$ is given by
	\begin{align*}
		V(P; \ga_n) =&\int_{\D R} \min_{a\in \ga_n}\|(x, 0)-a\|^2dP(x)\\
		=&\frac{1}{b-a}\Big(\int_a^{(m^2+1)\frac {(a_{1}+a_2)}{2}+mc} \rho((x, 0), (a_1, ma_1+c)) \, dx\\
		&+\sum _{i=2}^{n-1} \int_{(m^2+1)\frac {(a_{i-1}+a_i)}{2}+mc}^{(m^2+1)\frac {(a_{i}+a_{i+1})}{2}+mc} \rho((x, 0), (a_i, ma_i+c)) \, dx \\
		& +\int_{(m^2+1)\frac {(a_{n-1}+a_n)}{2}+mc}^b \rho((x, 0), (a_n, ma_n+c)) \, dx\Big).
	\end{align*}
	Since $V(P; \ga_n)$ gives the optimal error and is always differentiable with respect to $a_i$ for $2\leq i\leq n-1$, we have $\frac{\pa}{\pa a_i} V(P; \ga_n)=0$
	yielding
	\[a_{i+1}-a_i=a_i-a_{i-1} \te{ for } 2\leq i\leq n-1\]
	implying
	\begin{equation} \label{eq200}
		a_2-a_1=a_3-a_2=\cdots =a_n-a_{n-1}=k
	\end{equation}
	for some real $k$. Due to the same reasoning as given in \ref{Word:Case1}, we have $\frac{\pa}{\pa a_1} V(P; \ga_n)=0$ and $\frac{\pa}{\pa a_n} V(P; \ga_n)=0$, i.e.,
	\[2 (a-c m)-3 a_1 \left(m^2+1\right)+a_2 \left(m^2+1\right)=0, \te{ and }\]
	 \[a_{n-1} \left(m^2+1\right)-3 a_n \left(m^2+1\right)+2 (b-c m)=0\]
	implying
	\begin{equation} \label{eq2011}
		a_1=\frac{a-c m}{1+m^2}+\frac k2~ \te{ and } ~a_n=\frac{b-c m}{m^2+1}-\frac k2.
	\end{equation}
	Now, we have
	\begin{align*}
		b-a=&(a_1-a)+\sum_{i=2}^{n}(a_{i}-a_{i-1})+(b-a_n)
		= \left(\frac{a-cm}{1+m^2}+\frac k2 -a\right)+(n-1)k\\
		&\hspace{5cm}+\left(b-\frac{b-cm}{1+m^2}+\frac k2\right),
	\end{align*}
	which implies $k=\frac {b-a}{n(1+m^2)}.$
	Putting $k=\frac {b-a}{n(1+m^2)}$, by the expressions given in \eqref{eq200} and \eqref{eq2011}, we deduce that
	\[a_i=\frac {2i-1}{2n(1+m^2)}(b-a)+\frac{a-cm}{1+m^2} \te{ for } 1\leq i\leq n.\]
	To obtain the quantization error $V_n$, we proceed as follows:
	
	Since the probability distribution $P$ is uniform on its support, Equation~\eqref{eq200} helps us to deduce that the distortion errors contributed by $a_2, a_3, \cdots, a_{n-1}$ in their own Voronoi regions are equal, i.e., each term in the sum
	\[\sum _{i=2}^{n-1} \int_{(m^2+1)\frac {(a_{i-1}+a_i)}{2}+m c}^{(m^2+1)\frac {(a_{i}+a_{i+1})}{2}+mc} \rho((x, 0), (a_i, ma_i+c)) \, dx\] has the same value.
	Now, putting the values of $a_i$ for $2\leq i\leq n$ in terms of $a_1$ and $k$, we have
	\begin{align*}
		&V(P; \ga_n)\\
		=&\int_{\D R} \min_{a\in \ga_n}\|(x, 0)-a\|^2dP(x)\\
		= &\frac 1{b-a}\Big(\int_a^{(m^2+1)\frac {(2a_1+k)}{2}+mc}\rho\Big((x, 0), (a_1, ma_1+c)\Big) \, dx\\
		&+(n-2)\int_{(m^2+1)\frac {(2a_1+k)}{2}+mc}^{(m^2+1)\frac {(2a_1+3k)}{2}+mc} \rho\Big((x,0), (a_1+k, m(a_1+k)+c)\Big) \, dx\\
		&+\int_{(m^2+1)\frac {(2a_1+k(2n-3))}{2}+mc}^b \rho\Big((x, 0), (a_1+k(n-1), m(a_1+k(n-1))+c)\Big) \, dx \Big).
	\end{align*}
	Upon simplification, and putting $a_1=\frac{b-a}{2 \left(m^2+1\right) n}+\frac{a-cm}{m^2+1}$ and $k=\frac{b-a}{\left(m^2+1\right) n}$ in the above expression, we have the quantization error as
	\begin{align*}
		V_n=&\frac 1{12 \left(m^2+1\right) n^3}\Big(-48 (a - b)^2 m^2 + (a - b) (a - b + 72 c m + 8 (11 a - 2 b) m^2) n \\
		&\qquad \qquad -
		12 (a - b) m (5 c + (4 a + b) m) n^2 + 12 (c + a m)^2 n^3\Big).
	\end{align*}
	Thus, the proof of the theorem is complete.
\end{proof}

\begin{remark} 
	In Theorem~\ref{sec2Theorem1}, the assumptions 
	\[\max\set{a, (m^2+1)d+mc}=a \te{ and } \min\set{b, (m^2+1) e+mc}=b\]
	are necessary to guarantee that the elements in the optimal sets of $n$-points lie on the line segment joining the points $(d, md+c)$ and $(e, me+c)$. For more details, please see Proposition~\ref{sec2Prop1}. 
\end{remark} 

Let us now give the following corollary.
\begin{cor} \label{cor1}
	Let $P$ be a Borel probability measure on $\D R^2$ such that $P$ is uniform on its support $\set{(x, y) \in \D R^2 : 0\leq x\leq 2 \te{ and } y=0}$. For $n\in \D N$ with $n\geq 2$, let $\ga_n$ be an optimal set of $n$-points for $P$ such that the elements in the optimal set lie on the line $y=\sqrt 3 x$ between the elements $(0, 0)$ and $(2, 2\sqrt 3)$. Then,
	\[\ga_n=\Big\{\Big(\frac{2i-1}{4n}, \frac{2i-1}{4n} \sqrt 3\Big) : 1\leq i\leq n\Big\} \te{ and } V_n=\left\{\begin{array}{cc}
\frac{49}{48} & \te{ if } n=2,\\[8 pt] 

\frac{36 n^2+49 n-144}{12 n^3} & \te{ if } n\geq 3.
\end{array}
\right.\]
\end{cor}
\begin{proof}
	Putting $a=0,\, b=2, \, m=\sqrt 3,\, c=0, \, d=0, \te{ and } e=2$ in Theorem~\ref{sec2Theorem1}, we see that
	\[\max\set{a, (m^2+1)d+mc}=0=a \te{ and } \min\set{b, (m^2+1) e+mc}=2=b.\] Hence, by Theorem~\ref{sec2Theorem1}, we obtain the optimal sets $\ga_n$ and the corresponding quantization errors $V_n$ as follows:
		\[\ga_n=\Big\{\Big(\frac{2i-1}{4n}, \frac{2i-1}{4n} \sqrt 3\Big) : 1\leq i\leq n\Big\} \te{ and } V_n=\left\{\begin{array}{cc}
\frac{49}{48} & \te{ if } n=2,\\[8 pt] 

\frac{36 n^2+49 n-144}{12 n^3} & \te{ if } n\geq 3.
\end{array}
\right.\]
	Thus, the proof of the corollary is complete.
\end{proof}

\begin{remark}
	If $m=0$, $c=0$, $d=a$ and $e=b$, then by Theorem~\ref{sec2Theorem1}, the optimal set of $n$-points is given by $ \ga_n:=\set{a+\frac {2i-1}{2n}(b-a) : 1\leq i\leq n}$, and the corresponding quantization error is
	$V_n:=V_n(P)=\frac{(a-b)^2}{12 n^2},$
	which is Theorem~2.1.1 in \cite{RR}. Thus, Theorem~\ref{sec2Theorem1} generalizes Theorem~2.1.1 in \cite{RR}.
\end{remark}

The following proposition plays an important role in finding the optimal sets of $n$-points.
\begin{prop} \label{sec2Prop1}
	Let $P$ be a Borel probability measure on $\D R^2$ such that $P$ is uniform on its support $\set{(x, y) \in \D R^2 : a\leq x\leq b \te{ and } y=0}$.  For $n\in \D N$ with $n\geq 2$, let $\ga_n:=\{(a_i,m a_i+c): 1\leq i \leq n\}$ be an optimal set of $n$-points for $P$ so that the elements in the optimal set lie on the line $L$ between the two elements $(d, md+c)$ and $(e, me+c)$, where $d, e\in \D R$ with $d<e$. Then, $(i)$ if $(m^2+1)d+mc>a$ (or $(m^2+1) e+mc<b$), let $N$ be the largest positive integer such that  
	\[d<\frac {1}{2N(1+m^2)}(b-a)+\frac{a-cm}{1+m^2}  \ \Big (\te{or } \frac {2N-1}{2N(1+m^2)}(b-a)+\frac{a-cm}{1+m^2}<e\Big).\]
	Then, for all $n\geq N+1$, the optimal sets $\ga_n$ always contain the end element $(d, md+c)$ (or $(e, me+c))$.  On the other hand, $(ii)$ if $(m^2+1)d+mc>a$ and $(m^2+1) e+mc<b$, let $N:=\max\set{N_1, N_2}$, where $N_1$ and $N_2$ are  the largest positive integers such that  
	\[d<\frac {1}{2N_1(1+m^2)}(b-a)+\frac{a-cm}{1+m^2} \te{ and } \frac {2N_2-1}{2N_2(1+m^2)}(b-a)+\frac{a-cm}{1+m^2}<e.\]
	Then, for all $n\geq N+1$, the optimal sets $\ga_n$ always contain the end elements $(d, md+c)$ and $(e, me+c)$.
\end{prop}

\begin{proof}
	Let $\ga_n:=\{(a_i,m a_i+c): 1\leq i \leq n\}$ be an optimal set of $n$-points for $P$ so that the elements in the optimal set lie on the line $L$ between the two elements $(d, md+c)$ and $(e, me+c)$, where $d, e\in \D R$ with $d<e$. Also notice that the perpendiculars on the line $L$ passing through the elements $(a_i, ma_i+c)$ intersect the support of $P$ at the elements $((m^2+1)a_i+mc, 0)$, respectively, where $1\leq i\leq n$. 
	By Theorem~\ref{sec2Theorem1}, we know that
	\[a_i=\frac {2i-1}{2n(1+m^2)}(b-a)+\frac{a-cm}{1+m^2} \te{ for } 1\leq i\leq n.\]
	Suppose that $(m^2+1)d+mc>a$. Let $n=N$ be the largest positive integer such that
	\begin{equation} \label{eq61}
		\begin{aligned}
		&(m^2+1)d+mc<(m^2+1)a_1+mc, \te{ i.e., }\\
		 &d<a_1, \te{ i.e., } d<\frac {1}{2N(1+m^2)}(b-a)+\frac{a-cm}{1+m^2}.
		\end{aligned}
	\end{equation}
	Notice that the sequence $\set{\frac {1}{2n(1+m^2)}(b-a)+\frac{a-cm}{1+m^2}}$ is strictly decreasing, and hence for all $n\geq N+1$, the optimal sets $\ga_n$ always contain the end element $(d, md+c)$. Suppose that $(m^2+1)e+mc<b$. Let $n=N$ be the largest positive integer such that
	\begin{equation}\label{eq62}
	\begin{aligned} 
     &(m^2+1)a_N+mc<(m^2+1)e+mc, \te{ i.e., } a_N<e, \te{ i.e., }\\
     &  \frac {2N-1}{2N(1+m^2)}(b-a)+\frac{a-cm}{1+m^2}<e.
	\end{aligned}
	\end{equation}
	Notice that the sequence $\set{ \frac {2n-1}{2n(1+m^2)}(b-a)+\frac{a-cm}{1+m^2}}$ is strictly increasing, and hence for all $n\geq N+1$, the optimal sets $\ga_n$ always contain the end element $(e, me+c)$.
	Next, suppose that $(m^2+1)d+mc>a$ and $(m^2+1) e+mc<b$. Choose $N_1$ and $N_2$ same as $N$ described in \eqref{eq61} and \eqref{eq62}, respectively. Let $N=\max\set{N_1, N_2}$. Then, for all $n\geq N+1$, the optimal sets $\ga_n$ always contain the end elements $(d, md+c)$ and $(e, me+c)$. Thus, the proof of the proposition is complete. 
\end{proof}

\begin{note}
	In the following, we state and prove two theorems: Theorem \ref{sec2Theorem2} and Theorem \ref{sec2Theorem3}. To facilitate the proofs in both the theorems, Proposition \ref{sec2Prop1} can be used. However, in the proof of Theorem~\ref{sec2Theorem2}, we have not used Proposition~\ref{sec2Prop1}; on the other hand, in the proof of Theorem \ref{sec2Theorem3}, we have used Proposition~\ref{sec2Prop1}.
\end{note}

\begin{theorem}\label{sec2Theorem2}
	Let $P$ be a Borel probability measure on $\D R^2$ such that $P$ is uniform on its support $\set{(x, y) \in \D R^2 : 0\leq x\leq 2 \te{ and } y=0}$. For $n\in \D N$, let $\ga_n:=\{(a_i,  1): 1\leq i \leq n\}$ be an optimal set of $n$-points for $P$ so that the elements in the optimal sets lie on the line $y=1$ between the two elements $(\frac 12, 1)$ and $(\frac 32, 1)$. Then, $\ga_1=\set{(1,1)}$, $\ga_2=\set{(\frac 12, 1), (\frac 32, 1)}$, and for $n\geq 3$, we have
	\[a_i=\left\{\begin{array}{ccc}
		\frac 12 &\te{ if } i=1,\\
		\frac 12+\frac{(i-1)}{(n-1)}  & \te{ if } 2\leq i\leq n-1,\\
		\frac 32 & \te { if } i=n,
	\end{array}
	\right. \]
	and the quantization error for $n$-points is given by  $V_n=\frac{25 n^2-50 n+26}{24 (n-1)^2}$.
\end{theorem}

\begin{proof}
	The proofs of  $\ga_1=\set{(1,1)}$ and $\ga_2=\set{(\frac 12, 1), (\frac 32, 1)}$ are routine. We just give the proof for $n\geq 3$.
	Let $\ga_n:=\set{(t, 1) : t=a_i \te{ for } 1\leq i\leq n}$ be an optimal set of $n$-points such that $\frac 12\leq  a_1<a_2<\cdots<a_{n-1}<a_n\leq \frac 32$. Notice that the boundary of the Voronoi region of the element $(a_1,1)$ intersects the support of $P$ at the elements $(0, 0)$ and $(\frac 12(a_1+a_2), 0)$, the boundary of the Voronoi region of $(a_n,1)$ intersects the support of $P$ at the elements $(\frac 12(a_{n-1}+a_n), 0)$ and $(2,0)$. On the other hand, the boundaries of the Voronoi regions of $(a_i,1)$ for $2\leq i \leq n-1$ intersect the support of $P$ at the elements $(\frac 12(a_{i-1}+a_{i}), 0)$ and $(\frac 12(a_{i}+a_{i+1}), 0)$. Thus, the distortion error due to the set $\ga_n$ is given by
	\begin{align*}
		&V(P; \ga_n)=\int_{\D R} \min_{a\in \ga_n}\|(x, 0)-a\|^2dP(x)\\
		&=\int_0^{\frac{1}{2} \left(a_1+a_2\right)} \frac{1}{2} \left(\left(x-a_1\right){}^2+1\right) \, dx+\sum _{i=2}^{n-1} \int_{\frac{1}{2} \left(a_{i-1}+a_i\right)}^{\frac{1}{2} \left(a_{i+1}+a_i\right)} \frac{1}{2} \left(\left(x-a_i\right){}^2+1\right) \, dx\\
		&\qquad +\int_{\frac{1}{2} \left(a_{n-1}+a_n\right)}^2 \frac{1}{2} \left(\left(x-a_n\right){}^2+1\right) \, dx.
	\end{align*}
	Since $V(P; \ga_n)$ gives the optimal error and is differentiable with respect to $a_i$ for $2\leq i\leq n-1$, we have $\frac{\pa}{\pa a_i} V(P; \ga_n)=0$ implying
	\[a_{i+1}-a_i=a_i-a_{i-1} \te{ for } 2\leq i\leq n-1.\]
	This yields the fact that
	\begin{equation} \label{eq0000}
		a_2-a_1=a_3-a_2=\cdots =a_n-a_{n-1}=k
	\end{equation}
	for some real number $0<k<1$. By the equations in $\eqref{eq0000}$, we see that all terms in the sum $ \sum _{i=2}^{n-1} \int_{\frac{1}{2} \left(a_{i-1}+a_i\right)}^{\frac{1}{2} \left(a_{i+1}+a_i\right)} \frac{1}{2} \left(\left(x-a_i\right){}^2+1\right) \, dx$ have the same value.
	Again, by the equations in \eqref{eq0000} we have
	\[a_2=k+a_1, \, a_3=2k+a_1, \cdots, a_n=(n-1)k+a_1.\]
	Hence,
	\begin{align*}
		V(P; \ga_n)&= \int_0^{\frac{1}{2} \left(2 a_1+k\right)} \frac{1}{2} \left(\left(x-a_1\right){}^2+1\right) \, dx\\
		&+(n-2) \int_{\frac{1}{2} \left(2 a_1+k\right)}^{\frac{1}{2} \left(2 a_1+3 k\right)} \frac{1}{2} \left(\left(x-\left(a_1+k\right)\right){}^2+1\right) \, dx\\
		&\qquad \qquad +\int_{\frac{1}{2} \left(2 a_1+k (2 n-3)\right)}^2 \frac{1}{2} \left(\left(x-\left(a_1+k (n-1)\right)\right){}^2+1\right) \, dx,
	\end{align*}
	which upon simplification yields
	\begin{align*} V(P; \ga_n)&=\frac{1}{24} \Big(-12 a_1 (k (n-1)-2) \left(a_1+k (n-1)-2\right) \\
		&\qquad \qquad -k (n-1) \left(4 k^2 n^2-8 k (k+3) n+3 (k+4)^2\right)+56\Big),
	\end{align*}
	which is minimum if $a_1=\frac 12$ and $k=\frac 1{n-1}$, and the minimum value is $\frac{25 n^2-50 n+26}{24 (n-1)^2}$.
	As $k=\frac 1{n-1}$ and $a_1=\frac 12$, using the expression \eqref{eq0000}, we obtain
	\[a_i=\left\{\begin{array}{ccc}
		\frac 12 &\te{ if } i=1,\\
		\frac 12+\frac{(i-1)}{(n-1)}  & \te{ if } 2\leq i\leq n-1,\\
		\frac 32 & \te { if } i=n,
	\end{array}
	\right.  \]
	with  quantization error   $V_n=\frac{25 n^2-50 n+26}{24 (n-1)^2}$.
	Thus, the proof of the theorem is complete.
\end{proof}

\begin{remark}
	Comparing Theorem~\ref{sec2Theorem2} with Proposition~\ref{sec2Prop1}, we have $a=0, \, b=2, \, m=0, \, c=1, d=\frac 12, \te{ and } e=\frac 32$, and so
	\[(m^2+1)d+mc=\frac 12 >a \te{ and } (m^2+1)e+mc=\frac 32<b.\]
	Let $n=N_1$ be the largest positive integer such that 
	\[d<\frac {1}{2N_1(1+m^2)}(b-a)+\frac{a-cm}{1+m^2},\]
	which is true if $N_1<2$, i.e., $N_1=1$. Let $n=N_2$ be the largest positive integer such that
	\[\frac {2N_2-1}{2N_2(1+m^2)}(b-a)+\frac{a-cm}{1+m^2}<e,\]
	which is true if $N_2<2$, i.e., $N_2=1$. Take $N=\max\set{N_1, N_2}$. Then, $N=1$. By Proposition~\ref{sec2Prop1}, we can conclude that for all $n\geq 2$, the optimal sets $\ga_n$ will contain the end elements $(\frac 12, 1)$ and $(\frac 32, 1)$, which is clearly true by Theorem~\ref{sec2Theorem2}.
\end{remark}

\begin{theorem}\label{sec2Theorem3}
	Let $P$ be a Borel probability measure on $\D R^2$ such that $P$ is uniform on its support $\set{(x, y) \in \D R^2 : 0\leq x\leq 2 \te{ and } y=0}$. For $n\in \D N$, let $\ga_n:=\{(a_i,  1): 1\leq i \leq n\}$ be an optimal set of $n$-points for $P$ so that the elements in the optimal set lie on the line $y=1$ between the two elements $(0, 1)$ and $(\frac {28}{15}, 1)$, i.e., $0\leq a_1<a_2<\cdots <a_n\leq \frac {28}{15}$.
	Then, $\ga_1=\set{(1,1)}$, and for $1\leq n\leq 7$,
	\[\ga_n=\Big\{\Big(\frac{2i-1}{n}, 1\Big) : 1\leq i\leq n\Big\}.\] On the other hand, for $n\geq 8$, we obtain
	\[a_i=\left\{\begin{array}{ccc}
		\frac{28(2i-1)}{15(2n-1)}  & \te{ if } 1\leq i\leq n-1,\\
		\frac{28}{15} & \te { if } i=n,
	\end{array}
	\right. \]
	and the quantization error for $n$-points is given by  $V_n=\frac{7 (5788 (n-1) n+3015)}{10125 (1-2 n)^2}$.
\end{theorem}
\begin{proof}
	Let $\ga_n:=\set{(t, 1) : t=a_i \te{ for } 1\leq i\leq n}$ be an optimal set of $n$-points such that $0\leq  a_1<a_2<\cdots<a_{n-1}<a_n\leq \frac {28}{15}$ for all $n\in \D N$. Using Proposition \ref{sec2Prop1}, it can be proved that for all $n\geq 8$, the optimal sets always contain the end element $\frac{28}{15}$, i.e., $a_n=\frac{28}{15}$ for all $n\geq 8$.
	The proofs of $\ga_1=\set{(1,1)}$, and for $1\leq n\leq 7$,
	\[\ga_n=\Big\{\Big(\frac{2i-1}{n}, 1\Big) : 1\leq i\leq n\Big\},\]
	are routine. Here we prove the optimal sets of $n$-points for all $n\geq 8$.
	Proceeding in the similar lines as given in the proof of Theorem~\ref{sec2Theorem2}, we have
	\begin{equation*} \label{eq0}
		a_2-a_1=a_3-a_2=\cdots =a_n-a_{n-1}=k
	\end{equation*}
	for some real $k$, which implies
	\[a_1=a_n-(n-1)k, \, a_2=a_n-(n-2)k, \cdots, a_{n-1}=a_n-k.\]
	Also, by using $\frac{\pa}{\pa a_1} V(P; \ga_n)=0$, we get $3a_1-a_2=0$, which implies that $a_1=\frac k2$. Now we have
	\[\frac k2=a_1=a_n-(n-1)k=\frac{28}{15}-(n-1)k,\]
	this yields $k=\frac{56}{15(2n-1)}$.
	Using $a_n=\frac{28}{15}$ and $k=\frac{56}{15(2n-1)}$, we get $a_i=\frac{28(2i-1)}{15(2n-1)}$ for $1\leq i \leq n-1$ with quantization error
	\begin{align*}
		V(P; \ga_n)=&\frac 12\left( \int_0^{\frac{1}{2} \left(2 a_n-k (2 n-3)\right)} \left(\left(x-\left(a_n-k (n-1)\right)\right)^2+1\right) \, dx\right.\\
		\qquad&+(n-2) \int_{\frac{1}{2} \left(2 a_n-3 k\right)}^{\frac{1}{2} \left(2 a_n-k\right)} \left(\left(x-\left(a_n-k\right)\right)^2+1\right) \, dx\\
		&\left.+\int_{\frac{1}{2} \left(2 a_n-k\right)}^2 \left(\left(x-a_n\right)^2+1\right) \, dx\right)\\
		=&\frac 1{24}\Big(12 k^2 n^2 a_n-24 k^2 n a_n+12 k^2 a_n-12 k n a_n^2+12 k a_n^2+24 a_n^2\\
		\qquad & -48 a_n-4 k^3 n^3+12 k^3 n^2-11 k^3 n+3 k^3+56\Big)\\
		=&\frac{7 (5788 (n-1) n+3015)}{10125 (1-2 n)^2}.
	\end{align*}
	This completes the proof.
\end{proof}

\section{Constrained quantization when the support lies on a circle and the optimal elements lie on another circle} \label{sec2}
Let $O(0, 0)$ be the origin of the Cartesian plane. Let $C$ be the unit circle given by the parametric equations:
\[C:=\set{(x, y) :  x=\cos t, \, y=\sin t \te{ for } 0\leq t\leq 2\pi}.\]
Let the positive direction of the $x$-axis cut the circle at the element $A_0$, i.e., $A_0$ is represented by the parametric value $t=0$. Let $s$ be the distance of an element on $C$ along the arc starting from the element $A_0$ in the counterclockwise direction. Then,
\[ds=\sqrt{\Big(\frac {dx}{dt}\Big)^2+\Big(\frac{dy}{dt}\Big)^2}\,dt=dt.\]
Let $P$ be a uniform distribution with support the unit circle $C$. Then, the  probability density function $f(x, y)$ for $P$ is given by
\[f(x, y)=\left\{\begin{array}{ccc}
	\frac 1 {2\pi} & \te{ if } (x, y) \in C,\\
	0  & \te{ otherwise}.
\end{array}\right.
\]
Thus, we have $dP(s)=P(ds)=f(x, y) ds=\frac 1{2\pi} dt$. Moreover, we parameterize points on the unit circle by the central angle $t \in [0,2\pi)$. This allows us to express distances and arc-based quantities in terms of angular separation, which will be used in the estimates that follow. 

Let $L$ be a concentric circle with $C$, and $L$ has radius $a$, i.e., the parametric representation of the circle $L$ is given by
\[L:=\set{(x, y) :  x=a\cos \gq, \, y=a \sin\gq \te{ for } 0\leq \gq\leq 2\pi}.\]
In this section, we determine the optimal sets of $n$-points and the $n$th constrained quantization errors for the uniform distribution $P$ on $C$ under the condition that the elements in an optimal set lie on the circle $L$. Let the line $OA_0$ cut the circle $L$ at the element $B_0$, i.e., $B_0$ is represented on the circle $L$ by the parameter $\gq=0$.

\begin{prop} \label{sec3prop1}
	Any element on the circle $L$ forms an optimal set of one-point with quantization error $V_1=1+a^2. $
\end{prop}
\begin{proof} Let $\ga:=\set{(a\cos \gq, a\sin \gq)}$, where $0\leq \gq\leq  2\pi$, form an optimal set of one-point. Then, the distortion error $V(P; \ga)$
	is given by
	\[V(P; \ga)=\int_C \frac 1{2\pi} \rho((\cos t, \sin t), (a \cos \gq, a \sin \gq))\, dt=1+a^2,\]
	which does not depend on $\gq$ for any $0\leq \gq\leq 2\pi$. Hence, any element on the circle $L$ forms an optimal set of one-point, and the quantization error for one-point is given by $V_1=1+a^2$.
\end{proof}

\begin{prop} \label{sec3prop2}
	A set of the form $\set{(a \cos\gq, a \sin\gq), (-a \cos\gq, - a \sin\gq)}$, where $0\leq \gq \leq  2\pi$, forms an optimal set of two-points with quantization error
	$V_2=1+a^2-\frac {4a}{\pi}$.
\end{prop}
\begin{proof}
	Let $\ga:=\set{(a\cos \gq_1, a\sin \gq_1), (a\cos \gq_2, a\sin \gq_2)}$, where $0\leq \gq_1<\gq_2\leq 2\pi$, form an optimal set of two-points. Notice that the boundary of the Voronoi regions of the two elements in the optimal set is the line joining the two points given by the parameters $\gq=\frac{\gq_1+\gq_2}{2}$ and $\gq=\pi+\frac{\gq_1+\gq_2}{2}$. Then, the distortion error is given by
	\begin{align*}
		V(P; \ga) =&\frac 1{2\pi}\Big(\int_{-\pi+\frac{\gq_1+\gq_2}{2}}^{\frac{\gq_1+\gq_2}{2}} \rho\Big((\cos t,\sin t), (a \cos \gq_1, a \sin \gq_1)\Big) \, dt\\
		& +\int_{\frac{\gq_1+\gq_2}{2}}^{\pi+\frac{\gq_1+\gq_2}{2} } \rho\Big((\cos t,\sin t), (a \cos \gq_2,  a \sin \gq_2)\Big) \, dt\Big)\\
		 =&\frac 1{2\pi}\Big(\int_{-\pi+\frac{\gq_1+\gq_2}{2}}^{\frac{\gq_1+\gq_2}{2}}\Big(1 + a^2 - 2 a \cos(t-\gq_1)\Big) \, dt\\
		&+\int_{\frac{\gq_1+\gq_2}{2}}^{\pi+\frac{\gq_1+\gq_2}{2} } \Big(1 + a^2 - 2 a \cos(\gq-\gq_2)\Big)\, dt\Big),
	\end{align*}
	which upon simplification yields that
	\[V(P; \ga)=\frac 1{2\pi}\Big((1+a^2) 2\pi-8a \sin \frac{\gq_2-\gq_1}{2}\Big).\]
	Since $0<\frac{\gq_2-\gq_1}{2}<\pi$, we can say that $V(P; \ga)$ is minimum if $\gq_2-\gq_1=\pi$. Thus, an optimal set of two-points is given by $\set{(a \cos\gq, a \sin\gq), (-a \cos\gq, - a \sin\gq)}$ for $0\leq \gq\leq 2\pi$ with constrained quantization error
	$V_2=1+a^2-\frac {4a}{\pi}$,
	which yields the proposition.
\end{proof}

\begin{theorem} \label{sec3theorem1}
	Let $\ga_n$ be an optimal set of $n$-points for the uniform distribution $P$ on the unit circle $C$ for $n\in\D N$ with $n\geq 3$. Then,
	\[\ga_n=\Big\{\Big(a \cos \frac{(2i-1)\pi}n, a \sin   \frac{(2i-1)\pi}n\Big) : i=1, 2, \cdots, n \Big  \} \] and the corresponding quantization error is given by $V_n=a^2+1-\frac{2an}{\pi} \sin  \frac{\pi }{n}.$
\end{theorem}

\begin{proof}
	Let $\ga_n:=\set{a_1, a_2, \cdots, a_n}$ be an optimal set of $n$-points for $P$ with $n\geq 3$ such that the elements in the optimal set lie on the circle $L$. Let the boundary of the Voronoi regions of $a_i$ cut the circle $L$, in fact also the circle $C$, at the elements given by the parameters $\gq_{i-1}$ and $\gq_i$, where $1\leq i\leq n$. Since the circles have rotational symmetry, without any loss of generality, we can assume that $\gq_0=0$, and $\gq_n=2\pi$. Then, each $a_i$ on $L$ has the parametric representation $\frac 12(\gq_{i-1}+\gq_{i})$ for $1\leq i\leq n$.
	Then, the quantization error for $n$-points is given by
	\begin{align*}
		V(P; \ga_n)&=\int_{C} \min_{u\in \ga_n}\rho((\cos  t, \sin t), u)\,dP(s)\\
		&=\sum_{i=1}^n\int_{\gq_{i-1}}^{\gq_i}\frac 1{2\pi} \rho\Big(((\cos  t, \sin t), (\frac 12\cos\frac{\gq_{i-1}+\gq_i}{2}, \frac 12\sin\frac{\gq_{i-1}+\gq_i}{2}))\Big) \,dt\\
		&=\sum_{i=1}^n\int_{\gq_{i-1}}^{\gq_i}\frac 1{2\pi} \Big(a^2-2 a \cos (-\frac{\gq_{i-1}}{2}-\frac{\gq_i}{2}+t)+1\Big) \,d\gq\\
		&=\sum_{i=1}^n \frac 1{2\pi}\Big((a^2+1)(\gq_i-\gq_{i-1})-4a\sin \frac{\gq_i-\gq_{i-1}}{2}\Big),
	\end{align*}
	upon simplification, which yields
	\begin{equation} \label{eq500} V(P; \ga_n)=a^2+1-\frac{2a}{\pi}\sum_{i=1}^n\sin \frac{\gq_i-\gq_{i-1}}{2}.
	\end{equation}
	Since $V(P; \ga_n)$ gives the optimal error and is differentiable with respect to $\gq_i$ for all $1\leq i\leq n-1$, we have $\frac{\pa}{\pa \gq_i} V(P; \ga)=0$.
	For $1\leq i\leq n-1$, the equations $\frac{\pa}{\pa \gq_i} V(P; \ga)=0$ imply that
	\[\cos \frac{\gq_i-\gq_{i-1}}{2}=\cos \frac{\gq_{i+1}-\gq_i}{2} \te{ yielding }\] \[\frac{\gq_i-\gq_{i-1}}{2}=\frac{\gq_{i+1}-\gq_i}{2}, \te{ or } \frac{\gq_i-\gq_{i-1}}{2}=2\pi-\frac{\gq_{i+1}-\gq_i}{2}.\]
	Without any loss of generality, for $1\leq i\leq n-1$ we can take $\frac{\gq_i-\gq_{i-1}}{2}=\frac{\gq_{i+1}-\gq_i}{2}$.
	This yields the fact that
	\[\gq_1-\gq_0=\gq_2-\gq_1=\gq_3-\gq_2=\cdots=\gq_n-\gq_{n-1}=\frac {2 \pi}{n}.\]
	Thus, we have $\gq_i =\frac {2 \pi i} n$ for $i=1,2, \cdots, n$. Hence, if $\ga_n:=\set{a_1, a_2, \cdots, a_n}$ is an optimal set of $n$-points, then
	\[a_i=\Big(a \cos \frac{(2i-1)\pi}n, a \sin   \frac{(2i-1)\pi}n\Big) \te{ for } i=1, 2, \cdots, n,\]
	and the quantization error for $n$-points, by \eqref{eq500}, is given by
	\begin{align*}
		V_n=V(P; \ga_n)=a^2+1-\frac{2an}{\pi} \sin  \frac{\pi }{n}.
	\end{align*}
	Thus, the proof of the theorem is complete.
\end{proof}

\section{Constrained quantization when the support lies on a chord of a circle and the optimal elements lie on the circle} \label{sec3}
Let $C$ be a circle with center $(0, 0)$ and radius one, i.e., the equation of the circle is
$x^2+y^2=1,$
whose parametric representations are $x=\cos\gq$ and $y=\sin \gq$, where  $0\leq \gq\leq 2\pi$. Thus, if $(\cos \gq, \sin \gq)$ is an element on the circle, then we will represent it by $\gq$.
Let $P$ be a Borel probability measure on $\D R^2$ such that $P$ has support a chord of the circle, and $P$ is uniform on its support.
We now investigate the optimal sets of $n$-points and the $n$th constrained quantization errors for all $n\in \D N$ so that the optimal elements lie on the circle. The two cases can happen as described in the following two subsections.
\subsection{Chord is a diameter of the circle} \label{subdia}
Without any loss of generality, let us consider the horizontal diameter as the support of $P$, i.e., the support of $P$ is the closed interval $\set{(x, y) : -1\leq x\leq 1,\, y=0}$. Then, the probability density function is given by
\[f(x, y)=\left\{\begin{array}{ccc}
	\frac 1 {2} & \te{ if } -1\leq x\leq 1 \te{ and } y=0,\\
	0  & \te{ otherwise}.
\end{array}\right.
\]
Recall Fact~\ref{fact}. Here we have  $dP(s)=P(ds)=P(dx)=dP(x)=f(x, 0)dx$. Let $\ga_n$ be an optimal set of $n$-points for any $n\geq 1$. We know that an optimal set of one-point always exists. For any $n\geq 2$, since the boundary of the Voronoi regions of any two optimal elements, in this case, passes through the center of the circle, from the geometry, we see that among $n$ Voronoi regions, only two Voronoi regions contain elements from the support of $P$, i.e., only two Voronoi regions have positive probability. Hence, the optimal sets of $n$-points exist only for $n=1$ and $n=2$, and they do not exist for any $n\geq 3$. 

We now calculate the optimal sets of one-point and the two-points in the following propositions:

\begin{prop1}
	Any element on the circle forms an optimal set of one-point with constrained quantization error $V_1=\frac{4}{3}$.
\end{prop1}

\begin{proof} Let $(\cos \gq, \sin \gq)$ be an element on the circle. Then, the distortion error for $P$ with respect to this element is given by
	\begin{align*}
		V(P; \set{(\cos \gq, \sin \gq)})=&\int_{-1}^1 \rho((x,0), (\cos \gq, \sin \gq)) \,dP(x)\\
		=&\frac 12\int_{-1}^1 \rho((x,0), (\cos \gq, \sin \gq))\, dx=\frac 43,
		\end{align*}
	which does not depend on $\gq$. Hence, any element on the circle forms an optimal set of one-point with constrained quantization error $V_1=\frac{4}{3}$.
\end{proof}

\begin{prop1}
	The set $\set{(-1, 0), (1, 0)}$ forms an optimal set of two-points with constrained  quantization error $V_2=\frac 13$.
\end{prop1}
\begin{proof}
	From the geometry, we see that the boundary of any two elements on the circle passes through the center of the circle. Thus, in an optimal set of two-points, one Voronoi region will contain the left half, and the other Voronoi region will contain the right half of the support of $P$. Hence, by the routine calculation, we can show that $\set{(-1, 0), (1, 0)}$ forms an optimal set of two-points with constrained quantization error
	\[V_2=\frac 12 \Big(\int_{-1}^0 \rho((x, 0), (-1, 0)) dx+\int_0^1 \rho((x, 0), (1, 0))dx\Big)=\frac{1}{3}.\]
	Thus, the proof of the proposition is complete.
\end{proof}

\subsection{Chord is not a diameter of the circle}
In this case, for definiteness sake, we investigate the optimal sets of $n$-points and the $n$th constrained quantization errors for a Borel probability measure $P$ on $\D R^2$ such that $P$ has support the chord $y=-\frac 12$ for $-\frac {\sqrt 3}{ 2} \leq x\leq \frac {\sqrt 3}{ 2}$, and $P$ is uniform there. Then, the probability density function for $P$ is given by
\[f(x, y)=\left\{\begin{array}{ccc}
	\frac 1 {\sqrt 3} & \te{ if } -\frac {\sqrt 3}{2}\leq x\leq \frac {\sqrt 3}{2} \te{ and } y=-\frac 12,\\
	0  & \te{ otherwise}.
\end{array}\right.
\]
Recall that the circle has rotational symmetry. Thus, for any other chord, the technique of finding the optimal sets of $n$-points and the $n$th constrained quantization errors will be similar. Here we have $dP(s)=P(ds)=P(dx)=dP(x)=f(x, -\frac 12)dx$, where $x$ varies over the line $y=-\frac 12$. The arc of the circle subtended by the chord is represented by $\gq$ for $\frac {7\pi}6\leq \gq\leq  \frac {11\pi} 6$. Moreover, the circle is geometrically symmetric with respect to the line $y=0$, and also the probability measure is symmetric with respect to the line $y=0$, i.e., if two intervals of the same length lie on the support of $P$ and are equidistant from the line $y=0$, then they have the same probability. In proving the results, we can use this symmetry of the circle.

\begin{figure}
	\begin{tikzpicture}[line cap=round,line join=round,>=triangle 45, x=1.0cm,y=1.0cm]
		\clip(-2.196611570247939,-2.190413223140491) rectangle (2.198429752066107,2.199173553719003);
		\draw [line width=0.59pt] (0.,0.) circle (2 cm);
		\draw [line width=0.59pt,dash pattern=on 1pt off 1pt] (0.,3.)-- (0.,-3.);
		\draw [line width=0.59 pt,dash pattern=on 1pt off 1pt] (-3.,0.)-- (3.,0.);
		\draw [line width=0.59 pt] (-1.73205,-1.)-- (1.73205,-1.);
		\begin{scriptsize}
			\draw [fill=ffqqqq] (0.,-2) circle (2.5pt);
		\end{scriptsize}
	\end{tikzpicture} \quad
	\begin{tikzpicture}[line cap=round,line join=round,>=triangle 45, x=1.0cm,y=1.0cm]
		\clip(-2.196611570247939,-2.190413223140491) rectangle (2.198429752066107,2.199173553719003);
		\draw [line width=0.59 pt] (0.,0.) circle (2.cm);
		\draw [line width=0.59pt,dash pattern=on 1pt off 1pt] (0.,3.)-- (0.,-3.);
		\draw [line width=0.59pt,dash pattern=on 1pt off 1pt] (-3.,0.)-- (3.,0.);
		\draw [line width=0.59pt] (-1.73205,-1.)-- (1.73205,-1.);
		\draw [line width=0.59pt,dash pattern=on 1pt off 1pt] (-1.30931, -1.51186)-- (0.,0.);
		\draw [line width=0.59pt,dash pattern=on 1pt off 1pt] (0.,0.)-- (1.30931, -1.51186);
		\begin{scriptsize}
			\draw [fill=ffqqqq] (-1.30931, -1.51186) circle (2.5pt);
			\draw [fill=ffqqqq] (1.30931, -1.51186) circle (2.5pt);
		\end{scriptsize}
	\end{tikzpicture} \quad
	\begin{tikzpicture}[line cap=round,line join=round,>=triangle 45, x=1.0cm,y=1.0cm]
		\clip(-2.196611570247939,-2.190413223140491) rectangle (2.198429752066107,2.199173553719003);
		\draw [line width=0.59 pt] (0.,0.) circle (2.cm);
		\draw [line width=0.59pt,dash pattern=on 1pt off 1pt] (0.,3.)-- (0.,-3.);
		\draw [line width=0.59pt,dash pattern=on 1pt off 1pt] (-3.,0.)-- (3.,0.);
		\draw [line width=0.59pt] (-1.73205,-1.)-- (1.73205,-1.);
		\draw [line width=0.59pt,dash pattern=on 1pt off 1pt] (-1.47077, -1.3553)-- (0.,0.);
		\draw [line width=0.59pt,dash pattern=on 1pt off 1pt] (0.,0.)-- (1.47077, -1.3553);
		\begin{scriptsize}
			\draw [fill=ffqqqq] (-1.47077, -1.3553) circle (2.5pt);
			\draw [fill=ffqqqq] (0., -2.) circle (2.5pt);
			\draw [fill=ffqqqq] (1.47077, -1.3553) circle (2.5pt);
		\end{scriptsize}
	\end{tikzpicture} 
	\begin{tikzpicture}[line cap=round,line join=round,>=triangle 45, x=1.0cm,y=1.0cm]
		\clip(-2.196611570247939,-2.190413223140491) rectangle (2.198429752066107,2.199173553719003);
		\draw [line width=0.59 pt] (0.,0.) circle (2.cm);
		\draw [line width=0.59pt,dash pattern=on 1pt off 1pt] (0.,3.)-- (0.,-3.);
		\draw [line width=0.59pt,dash pattern=on 1pt off 1pt] (-3.,0.)-- (3.,0.);
		\draw [line width=0.59pt] (-1.73205,-1.)-- (1.73205,-1.);
		\draw [line width=0.59pt,dash pattern=on 1pt off 1pt] (-1.54394, -1.27132)-- (0.,0.);
		\draw [line width=0.59pt,dash pattern=on 1pt off 1pt] (0.,0.)-- (-0.658041, -1.88865);
		\draw [line width=0.59pt,dash pattern=on 1pt off 1pt] (0.,0.)-- (1.54394, -1.27132);
		\draw [line width=0.59pt,dash pattern=on 1pt off 1pt] (0.,0.)-- (0.658041, -1.88865);
		\begin{scriptsize}
			\draw [fill=ffqqqq] (-1.54394, -1.27132) circle (2.5pt);
			\draw [fill=ffqqqq] (-0.658041, -1.88865) circle (2.5pt);
			\draw [fill=ffqqqq] (1.54394, -1.27132) circle (2.5pt);
			\draw [fill=ffqqqq] (0.658041, -1.88865) circle (2.5pt);
		\end{scriptsize}
	\end{tikzpicture} \quad
	\begin{tikzpicture}[line cap=round,line join=round,>=triangle 45, x=1.0cm,y=1.0cm]
		\clip(-2.196611570247939,-2.190413223140491) rectangle (2.198429752066107,2.199173553719003);
		\draw [line width=0.59 pt] (0.,0.) circle (2.cm);
		\draw [line width=0.59pt,dash pattern=on 1pt off 1pt] (0.,3.)-- (0.,-3.);
		\draw [line width=0.59pt,dash pattern=on 1pt off 1pt] (-3.,0.)-- (3.,0.);
		\draw [line width=0.59pt] (-1.73205,-1.)-- (1.73205,-1.);
		\draw [line width=0.59pt,dash pattern=on 1pt off 1pt] (-1.58533, -1.21931)-- (0.,0.);
		\draw [line width=0.59pt,dash pattern=on 1pt off 1pt] (0.,0.)-- (-0.984888, -1.74069);
		\draw [line width=0.59pt,dash pattern=on 1pt off 1pt] (1.58533, -1.21931)-- (0.,0.);
		\draw [line width=0.59pt,dash pattern=on 1pt off 1pt] (0.,0.)-- (0.984888, -1.74069);
		\begin{scriptsize}
			\draw [fill=ffqqqq] (-1.58533, -1.21931) circle (2.5pt);
			\draw [fill=ffqqqq] (-0.984888, -1.74069) circle (2.5pt);
			\draw [fill=ffqqqq] (0.,-2.) circle (2.5pt);
			\draw [fill=ffqqqq] (1.58533, -1.21931) circle (2.5pt);
			\draw [fill=ffqqqq] (0.984888, -1.74069) circle (2.5pt);
		\end{scriptsize}
	\end{tikzpicture}\quad
	\begin{tikzpicture}[line cap=round,line join=round,>=triangle 45, x=1.0cm,y=1.0cm]
		\clip(-2.196611570247939,-2.190413223140491) rectangle (2.198429752066107,2.199173553719003);
		\draw [line width=0.59 pt] (0.,0.) circle (2.cm);
		\draw [line width=0.59pt,dash pattern=on 1pt off 1pt] (0.,3.)-- (0.,-3.);
		\draw [line width=0.59pt,dash pattern=on 1pt off 1pt] (-3.,0.)-- (3.,0.);
		\draw [line width=0.59pt] (-1.73205,-1.)-- (1.73205,-1.);
		\draw [line width=0.59pt,dash pattern=on 1pt off 1pt] (-1.61187, -1.184)-- (0,0);
		\draw [line width=0.59pt,dash pattern=on 1pt off 1pt] (-1.1689, -1.62286)-- (0,0);
		\draw [line width=0.59pt,dash pattern=on 1pt off 1pt] (-0.438876, -1.95125)-- (0,0);
		\draw [line width=0.59pt,dash pattern=on 1pt off 1pt] (1.61187, -1.184)-- (0,0);
		\draw [line width=0.59pt,dash pattern=on 1pt off 1pt] (1.1689, -1.62286)-- (0,0);
		\draw [line width=0.59pt,dash pattern=on 1pt off 1pt] (0.438876, -1.95125)-- (0,0);
		\begin{scriptsize}
			\draw [fill=ffqqqq] (-1.61187, -1.184) circle (2.5pt);
			\draw [fill=ffqqqq] (-1.1689, -1.62286) circle (2.5pt);
			\draw [fill=ffqqqq] (-0.438876, -1.95125) circle (2.5pt);
			\draw [fill=ffqqqq] (1.61187, -1.184) circle (2.5pt);
			\draw [fill=ffqqqq] (1.1689, -1.62286) circle (2.5pt);
			\draw [fill=ffqqqq] (0.438876, -1.95125) circle (2.5pt);
		\end{scriptsize}
	\end{tikzpicture}\\ 
	\noindent \begin{tikzpicture}[line cap=round,line join=round,>=triangle 45, x=1.0cm,y=1.0cm]
		\clip(-2.196611570247939,-2.190413223140491) rectangle (2.198429752066107,2.199173553719003);
		\draw [line width=0.59 pt] (0.,0.) circle (2.cm);
		\draw [line width=0.59pt,dash pattern=on 1pt off 1pt] (0.,3.)-- (0.,-3.);
		\draw [line width=0.59pt,dash pattern=on 1pt off 1pt] (-3.,0.)-- (3.,0.);
		\draw [line width=0.59pt] (-1.73205,-1.)-- (1.73205,-1.);
		\draw [line width=0.59pt,dash pattern=on 1pt off 1pt] (-1.63032, -1.15847)-- (0.,0.);
		\draw [line width=0.59pt,dash pattern=on 1pt off 1pt] (0.,0.)-- (-1.2839, -1.5335);
		\draw [line width=0.59pt,dash pattern=on 1pt off 1pt] (0.,0.)-- (-0.726701, -1.86331);
		\draw [line width=0.59pt,dash pattern=on 1pt off 1pt] (1.63032, -1.15847)-- (0.,0.);
		\draw [line width=0.59pt,dash pattern=on 1pt off 1pt] (0.,0.)-- (1.2839, -1.5335);
		\draw [line width=0.59pt,dash pattern=on 1pt off 1pt] (0.,0.)-- (0.726701, -1.86331);
		\begin{scriptsize}
			\draw [fill=ffqqqq] (-1.63032, -1.15847) circle (2.5pt);
			\draw [fill=ffqqqq] (-1.2839, -1.5335) circle (2.5pt);
			\draw [fill=ffqqqq] (-0.726701, -1.86331) circle (2.5pt);
			\draw [fill=ffqqqq] (0.,-2.) circle (2.5pt);
			\draw [fill=ffqqqq] (1.63032, -1.15847) circle (2.5pt);
			\draw [fill=ffqqqq] (1.2839, -1.5335) circle (2.5pt);
			\draw [fill=ffqqqq] (0.726701, -1.86331) circle (2.5pt);
		\end{scriptsize}
	\end{tikzpicture} \quad
	\begin{tikzpicture}[line cap=round,line join=round,>=triangle 45, x=1.0cm,y=1.0cm]
		\clip(-2.196611570247939,-2.190413223140491) rectangle (2.198429752066107,2.199173553719003);
		\draw [line width=0.59 pt] (0.,0.) circle (2.cm);
		\draw [line width=0.59pt,dash pattern=on 1pt off 1pt] (0.,3.)-- (0.,-3.);
		\draw [line width=0.59pt,dash pattern=on 1pt off 1pt] (-3.,0.)-- (3.,0.);
		\draw [line width=0.59pt] (-1.73205,-1.)-- (1.73205,-1.);
		\draw [line width=0.59pt,dash pattern=on 1pt off 1pt] (-1.64388, -1.13915)-- (0.,0.);
		\draw [line width=0.59pt,dash pattern=on 1pt off 1pt] (0.,0.)-- (-1.36154, -1.465);
		\draw [line width=0.59pt,dash pattern=on 1pt off 1pt] (0.,0.)-- (-0.9216, -1.77501);
		\draw [line width=0.59pt,dash pattern=on 1pt off 1pt] (0.,0.)-- (-0.329196, -1.97272);
		\draw [line width=0.59pt,dash pattern=on 1pt off 1pt] (1.64388, -1.13915)-- (0.,0.);
		\draw [line width=0.59pt,dash pattern=on 1pt off 1pt] (0.,0.)-- (1.36154, -1.465);
		\draw [line width=0.59pt,dash pattern=on 1pt off 1pt] (0.,0.)-- (0.9216, -1.77501);
		\draw [line width=0.59pt,dash pattern=on 1pt off 1pt] (0.,0.)-- (0.329196, -1.97272);
		\begin{scriptsize}
			\draw [fill=ffqqqq] (-1.64388, -1.13915) circle (2.5pt);
			\draw [fill=ffqqqq] (-1.36154, -1.465) circle (2.5pt);
			\draw [fill=ffqqqq] (-0.9216, -1.77501) circle (2.5pt);
			\draw [fill=ffqqqq] (-0.329196, -1.97272) circle (2.5pt);
			\draw [fill=ffqqqq] (1.64388, -1.13915) circle (2.5pt);
			\draw [fill=ffqqqq] (1.36154, -1.465) circle (2.5pt);
			\draw [fill=ffqqqq] (0.9216, -1.77501) circle (2.5pt);
			\draw [fill=ffqqqq] (0.329196, -1.97272) circle (2.5pt);
		\end{scriptsize}
	\end{tikzpicture}
\end{figure}

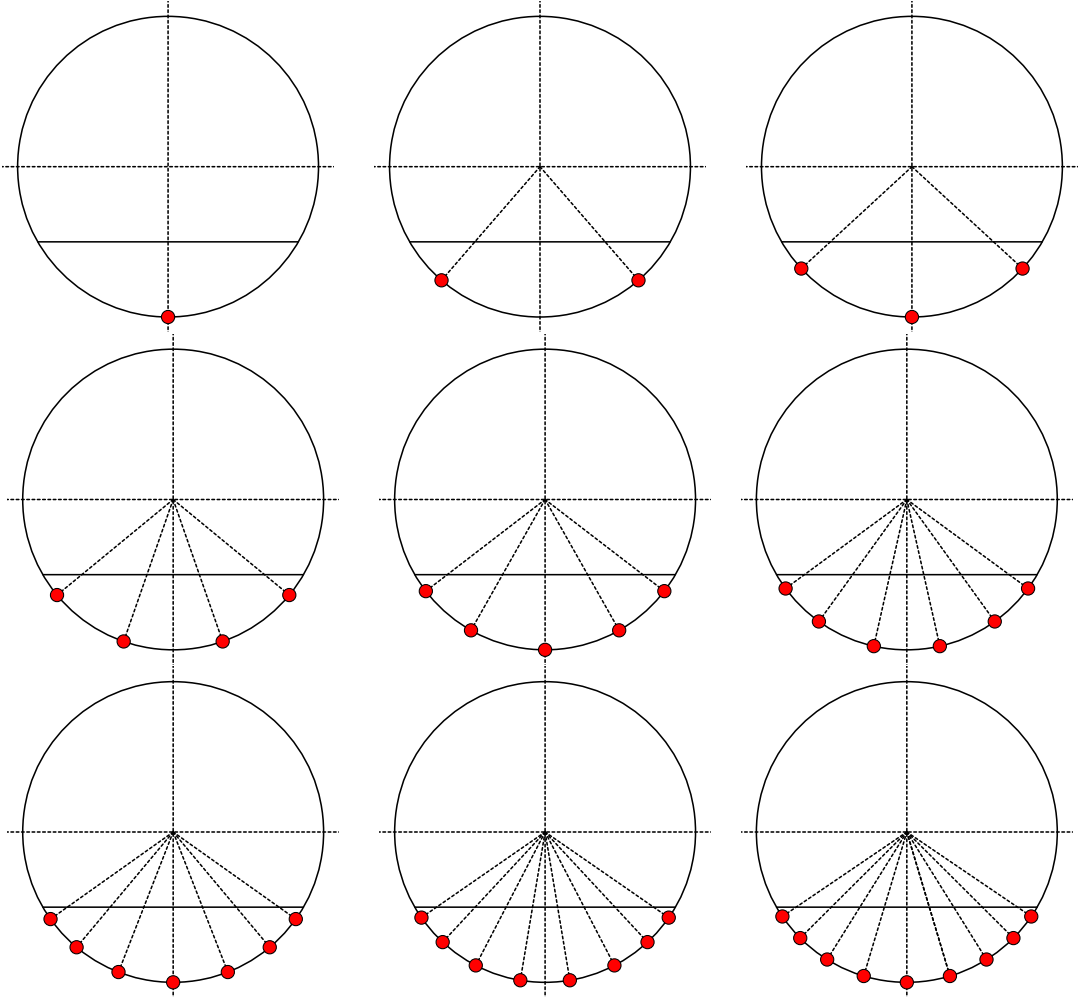
\begin{figure}
		\begin{tikzpicture}[line cap=round,line join=round,>=triangle 45, x=1.0cm,y=1.0cm]
		\clip(-2.196611570247939,-2.190413223140491) rectangle (2.198429752066107,2.199173553719003);
		\draw [line width=0.59 pt] (0.,0.) circle (2.cm);
		\draw [line width=0.59pt,dash pattern=on 1pt off 1pt] (0.,3.)-- (0.,-3.);
		\draw [line width=0.59pt,dash pattern=on 1pt off 1pt] (-3.,0.)-- (3.,0.);
		\draw [line width=0.59pt] (-1.73205,-1.)-- (1.73205,-1.);
		\draw [line width=0.59pt,dash pattern=on 1pt off 1pt] (-1.65425, -1.12403)-- (0, 0);
		\draw [line width=0.59pt,dash pattern=on 1pt off 1pt] (-1.41706, -1.41136)-- (0, 0);
		\draw [line width=0.59pt,dash pattern=on 1pt off 1pt] (-1.05905, -1.69659)-- (0, 0);
		\draw [line width=0.59pt,dash pattern=on 1pt off 1pt] (-0.572988, -1.91616)-- (0, 0);
		\draw [line width=0.59pt,dash pattern=on 1pt off 1pt] (1.65425, -1.12403)-- (0, 0);
		\draw [line width=0.59pt,dash pattern=on 1pt off 1pt] (1.41706, -1.41136)-- (0, 0);
		\draw [line width=0.59pt,dash pattern=on 1pt off 1pt] (0.572988, -1.91616)-- (0, 0);
		\draw [line width=0.59pt,dash pattern=on 1pt off 1pt] (1.05905, -1.69659)-- (0, 0);
		\draw [line width=0.59pt,dash pattern=on 1pt off 1pt] (0.572988, -1.91616)-- (0, 0);
		\begin{scriptsize}
			\draw [fill=ffqqqq] (-1.65425, -1.12403) circle (2.5pt);
			\draw [fill=ffqqqq] (-1.41706, -1.41136) circle (2.5pt);
			\draw [fill=ffqqqq] (-1.05905, -1.69659) circle (2.5pt);
			\draw [fill=ffqqqq] (-0.572988, -1.91616) circle (2.5pt);
			\draw [fill=ffqqqq] (0., -2.) circle (2.5pt);
			\draw [fill=ffqqqq] (1.65425, -1.12403) circle (2.5pt);
			\draw [fill=ffqqqq] (1.41706, -1.41136) circle (2.5pt);
			\draw [fill=ffqqqq] (1.05905, -1.69659) circle (2.5pt);
			\draw [fill=ffqqqq] (0.572988, -1.91616) circle (2.5pt);
		\end{scriptsize}
	\end{tikzpicture}
	\caption{Optimal configuration of $n$ elements for $1\leq n\leq 9$.} \label{Fig1}
	\end{figure}

\begin{prop1} The set $\set{(0, -1)}$ forms an optimal set of one-point with quantization error $V_1=\frac 12$.
\end{prop1}
\begin{proof}
	Let us consider an element $(\cos \gq, \sin \gq)$ on the circle. The distortion error for $P$ with respect to the set $\set{(\cos \gq, \sin \gq)}$ is given by
	\[V(P; \set{(\cos \gq, \sin \gq)})=\int_{-\frac{\sqrt{3}}{2}}^{\frac{\sqrt{3}}{2}} \frac 1 {\sqrt 3}\rho((x,-\frac{1}{2}), (\cos \gq,\sin \gq)) \, dx=\sin \gq+\frac{3}{2},\]
	the minimum value of which is $\frac 12$ and it occurs when $\gq=\frac{3\pi}2$ (see Figure~\ref{Fig1}).
	Thus, the proof of the proposition is yielded.
\end{proof}
\begin{prop1} The optimal set of two-points is given by
	\[\Big\{2 \pi -2 \tan ^{-1}\Big(\frac{\sqrt{3}}{2}+\frac{\sqrt{7}}{2}\Big), \pi +2 \tan ^{-1}\Big(\frac{\sqrt{3}}{2}+\frac{\sqrt{7}}{2}\Big)\Big\}\] with quantization error $V_2=\frac{1}{2} \left(3-\sqrt{7}\right)$.
\end{prop1}
\begin{proof}
	Since the probability measure is symmetric with respect to the line $y=0$, we can assume that in an optimal set of two-points, the Voronoi region of one element will contain the left half of the chord, and the Voronoi region of the other element will contain the right half of the chord, i.e., the boundary of the two Voronoi regions is the $y$-axis.
	Let the left element is $(\cos \gq, \sin\gq)$. Then, due to symmetry, the distortion error for the two elements is given by
	\[2\int_{-\frac{\sqrt{3}}{2}}^0 \frac 1 {\sqrt 3}\rho((x,-\frac{1}{2}), (\cos \gq,\sin \gq)) \, dx=\sin \gq+\frac{1}{2} \sqrt{3} \cos \gq+\frac{3}{2},\]
	which is minimum if $\gq=2 \pi -2 \tan ^{-1} (\frac{\sqrt{3}}{2}+\frac{\sqrt{7}}{2})$, and the minimum value is $\frac{1}{2} \left(3-\sqrt{7}\right)$. Thus, the one element is represented by $\gq=2 \pi -2 \tan ^{-1}(\frac{\sqrt{3}}{2}+\frac{\sqrt{7}}{2})$, and due to symmetry the other element is represented by $\gq=\pi+2 \tan ^{-1}(\frac{\sqrt{3}}{2}+\frac{\sqrt{7}}{2})$ with quantization error for two-points $V_2=\frac{1}{2} \left(3-\sqrt{7}\right)$ (see Figure~\ref{Fig1}). Thus, the proof of the proposition is complete.
\end{proof}

\begin{remark1} \label{rem111}
	Due to the symmetry of the probability measure $P$ and the geometrical symmetry of the circle, we can assume that in an optimal set of $n$-points, where $n\geq 3$, if $n$ is even, then there are $\frac n 2$ elements to the left of the $y$-axis and $\frac n2$ elements to the right of the $y$-axis. On the other hand, if $n$ is odd, then there are $\frac {n-1} 2$ elements to the left of the $y$-axis and $\frac {n-1}2$ elements to the right of the $y$-axis, and the remaining one element will be the element $(-1, 0)$. Moreover, whether $n$ is even or odd, the set of elements on the left side and the set of elements on the right side are reflections of each other with respect to the $y$-axis. Due to this fact, in the sequel of this section, we calculate the optimal sets of $n$-points for $n=8$ and $n=9$. Following the similar technique, whether $n$ is even or odd, one can calculate the locations of the elements for any positive integer $n\geq 3$.
\end{remark1}

\begin{prop1} \label{prop11110}  The optimal set of eight-points is given by
	\begin{align*}
		&\{(-0.821938,-0.569577),(-0.680768,-0.732499),(-0.4608,-0.887504),\\
		&(-0.164598,-0.986361), (0.821938,-0.569577),(0.680768,-0.732499),\\
		&(0.4608,-0.887504),(0.164598,-0.986361)\}
	\end{align*} with quantization error $V_8=0.12327$.
\end{prop1}
\begin{proof}
	Let $\ga_8:=\set{\gq_1, \gq_2, \cdots\gq_8}$ be an optimal set of eight-points. Without any loss of generality, we can assume that $\gq_1<\gq_2<\cdots <\gq_8$. Due to symmetry as mentioned in Remark~\ref{rem111}, the boundary of the Voronoi regions of $\gq_4$ and $\gq_5$ is the $y$-axis, and the elements on the right side of $y$-axis are the reflections of the elements on the left side of $y$-axis with respect to the $y$-axis, and vice versa. Thus, it is enough to calculate the first four elements $\gq_1, \gq_2, \gq_3, \gq_4$. Let the boundaries of the Voronoi regions of $\gq_i$ and $\gq_{i+1}$ intersect the support of $P$ at the elements $(a_i, -\frac 12)$, where $1\leq i\leq 3$. Because of the symmetry, the distortion error is given by
	\begin{align} \label{chord11}
		V(P; \ga_8)=&2\Big(\int_{-\frac{\sqrt{3}}2}^{a_1}\rho((x, 0), (\cos \gq_1, \sin \gq_1))\, dP(x)\notag\\
		&\qquad+\sum_{i=1}^2\int_{a_i}^{a_{i+1}}\rho((x, 0), (\cos \gq_{i+1}, \sin \gq_{i+1}))\, dP(x)\\
		&\qquad +\int_{a_3}^{0}\rho((x, 0), (\cos \gq_4, \sin \gq_4))\, dP(x)\Big) \notag.
	\end{align}
	The canonical equations  are
	\begin{equation*} \label{chord}
		\rho((a_i, -\frac 12), (\cos \gq_i, \sin \gq_i))- \rho((a_i, -\frac 12), (\cos \gq_{i+1}, \sin \gq_{i+1}))=0 \te{ for } i=1, 2, 3.
	\end{equation*}
	Solving the canonical equations, we have
	\[a_1= \frac{\sin \gq_1-\sin \gq_2}{2 \left(\cos \gq_1-\cos \gq_2\right)}, a_2= \frac{\sin \gq_2-\sin \gq_3}{2 \left(\cos \gq_2-\cos \gq_3\right)}, a_3= \frac{\sin \gq_3-\sin \gq_4}{2 \left(\cos \gq_3-\cos \gq_4\right)}.\]
	Putting the values of $a_1, a_2, a_3$ in \eqref{chord11}, we see that $V(P; \ga_8)$ is a function of $\gq_i$ for $i=1, 2,3, 4$.
	Since $V(P; \ga_8)$ is optimal, we have
	\[\frac{\pa}{\pa \gq_i} V(P; \ga_8)=0 \te{ for } i=1, 2, 3, 4.\]
	Solving the above four equations, we obtain the values of $\gq_i$ for which $V(P; \ga_8)$ is minimum as
	\[\gq_1=3.74758, \gq_2= 3.96358, \gq_3= 4.23349, \gq_4= 4.54704.\]
	Due to symmetry $\gq_5, \gq_6, \gq_7, \gq_8$ can also be obtained. Recall that $\gq_i$ represents the element $(\cos \gq_i, \sin \gq_i)$.
	Thus, we obtain the optimal set of eight-points as mentioned in the proposition with quantization error $V_8=0.12327$ (see Figure~\ref{Fig1}). Thus, the proof of the proposition is complete.
\end{proof}

\begin{prop1} The optimal set of nine-points is given by
	\begin{align*}
		&\{(-0.827126, -0.562016), (-0.708531, -0.70568), (-0.529525, -0.848294),\\
		 &(-0.286494, -0.958082), (0., -1), (0.827126, -0.562016),(0.708531, -0.70568),\\
		 & (0.529525, -0.848294), (0.286494, -0.958082)\}
	\end{align*} with quantization error $V_9=0.122546$.
\end{prop1}
\begin{proof}
	Recall Remark~\ref{rem111}. We can assume that the optimal set of nine-points is
	$\ga_9=\set{\gq_i : 1\leq i\leq 9}$ such that $\gq_i<\gq_{i+1}$ for $1\leq i\leq 8$, where $\gq_5=\frac{3\pi}2$.  Because of the same reasoning as given in the proof of Proposition~\ref{prop11110}, we have the distortion error as
	\begin{equation}\label{chord111}
	\begin{aligned} 
		V(P; \ga_9)=&2\Big(\int_{-\frac{\sqrt{3}}2}^{a_1}\rho((x, 0), (\cos \gq_1, \sin \gq_1))\, dP(x)\notag\\
		 &+\sum_{i=1}^3\int_{a_i}^{a_{i+1}}\rho((x, 0), (\cos \gq_{i+1}, \sin \gq_{i+1}))\, dP(x)\notag\\
		&+\int_{a_4}^{0}\rho((x, 0), (0, -1))\, dP(x)\Big).
	\end{aligned}
	\end{equation}
	The canonical equations are
	\begin{equation*} \label{chord1}
		\rho((a_i, -\frac 12), (\cos \gq_i, \sin \gq_i))- \rho((a_i, -\frac 12), (\cos \gq_{i+1}, \sin \gq_{i+1}))=0 \te{ for } i=1, 2, 3, 4.\end{equation*}
	Solving the canonical equations, we obtain the values of $a_i$ for $1\leq i\leq 4$. Putting the values of $a_i$ in \eqref{chord111}, we see that $V(P; \ga_9)$ is a function of $\gq_i$ for $i=1, 2,3, 4$.
	Since $V(P; \ga_9)$ is optimal we have
	\[\frac{\pa}{\pa \gq_i} V(P; \ga_9)=0 \te{ for } i=1, 2, 3, 4.\]
	Solving the above four equations, we obtain the values of $\gq_i$ for which $V(P; \ga_9)$ is minimum as
	\[\gq_1=3.73841,\gq_2= 3.92497, \gq_3=4.15435, \gq_4=4.42182,\]
	Due to symmetry $\gq_6, \gq_7, \gq_8, \gq_9$ can also be obtained. Recall that $\gq_i$ represents the element $(\cos \gq_i, \sin \gq_i)$.
	Hence, we obtain the optimal set of nine-points as mentioned in the proposition with quantization error $V_9=0.122546$ (see Figure~\ref{Fig1}). Thus, the proof of the proposition is complete.
\end{proof}

\section{Quantization dimensions and quantization coefficients}  \label{sec5}
Let $P$ be a Borel probability measure on $\D R^k$ equipped with a metric, and let $ r\in (0,\infty)$. A quantization without any constraint is called an unconstrained quantization. In unconstrained quantization (see \cite{GL}), the numbers
\begin{equation} \label{eq55} \ul D_r(P):=\liminf_{n\to \infty}  \frac{r\log n}{-\log V_{n,r}(P)} \te{ and } \ol D_r(P):=\limsup_{n\to \infty} \frac{r\log n}{-\log V_{n, r}(P)}, \end{equation}
are called the \tit{lower} and the \tit{upper quantization dimensions} of the probability measure $P$ of order $r$, respectively. If $\ul D_r (P)=\ol D_r (P)$, the common value is called the \tit{quantization dimension} of $P$ of order $r$ and is denoted by $D_r(P)$. In unconstrained quantization (see \cite{GL}) for any $\gk>0$, the two numbers $\liminf_n n^{\frac r \gk}  V_{n, r}(P)$ and $\limsup_n  n^{\frac r \gk}V_{n, r}(P)$ are, respectively, called the \tit{$\gk$-dimensional lower} and the \tit{upper quantization coefficients} for $P$. The quantization coefficients provide us with more accurate information about the asymptotics of the quantization error than the quantization dimension. In unconstrained case, it is known that for an absolutely continuous probability measure, the quantization dimension always exists and equals the Euclidean dimension of the support of $P$, and the quantization coefficient exists as a finite positive number (see \cite{BW}). If the $\gk$-dimensional lower and the upper quantization coefficients for $P$ are finite and  positive, then $\gk$ equals the quantization dimension of $P$. 

\begin{remark} \label{remark1} 
	Unconstrained quantization error $V_{n, r}(P)$ goes to zero as $n$ tends to infinity (see \cite{GL}). This is not true in the case of constrained quantization. Constrained quantization error $V_{n, r}(P)$ can approach to any nonnegative number  as $n$ tends to infinity, and it depends on the constraint $S$ that occurs in the definition of constrained quantization error as given in \eqref{Vr}. In this regard, below we give some examples.  
\end{remark} 

Let $P$ be a Borel probability measure on $\D R^2$ such that $P$ is uniform on its support $\set{(x, y) \in \D R^2 : 0\leq x\leq 2 \te{ and } y=0}$.
Let $V_n(P):=V_{n, 2}(P)$ be its constrained quantization error. If the elements in the optimal sets lie on the line $y=1$ between the two elements $(\frac 12, 1)$ and $(\frac 32, 1)$, then by Theorem~\ref{sec2Theorem2}, for $n\geq 3$,
\begin{equation} \label{eq56} V_n(P)=\frac{25 n^2-50 n+26}{24 (n-1)^2} \te{ implying } \lim_{n\to \infty} V_n(P)=\frac{25}{24}.
\end{equation}
If the elements in the optimal sets lie on the line $y=1$ between the two elements $(0, 1)$ and $(\frac {28}{15}, 1)$, then by Theorem~\ref{sec2Theorem3}, for $n\geq 8$,
\begin{equation} \label{eq57}V_n(P)=\frac{7 (5788 (n-1) n+3015)}{10125 (1-2 n)^2}\te{ implying } \lim_{n\to \infty} V_n(P)=\frac{10129}{10125}.\end{equation}
On the other hand, {if the elements in the optimal sets lie on the line $y=\sqrt 3 x$ between the two elements $(0, 0)$ and $(2, 2\sqrt 3)$, then by Corollary~\ref{cor1}, for $n\geq 2$,
	\begin{equation} \label{eq571}V_n(P)=\frac{36 n^2+49 n-144}{12 n^3} \te{ implying } \lim_{n\to \infty} V_n(P)=0.\end{equation}
	Moreover, notice that if $P$ is a uniform distribution on a unit circle, and if the elements in an optimal set of $n$-points lie on a concentric circle with radius $a$, then by Theorem~\ref{sec3theorem1}, for $n\geq 3$,
	\begin{equation} \label{eq58} V_n(P)=a^2+1-\frac{2an}{\pi} \sin  \frac{\pi }{n} \te{ implying } \lim_{n\to \infty} V_n(P)=(a-1)^2,
	\end{equation}
	which is a nonnegative constant depending on the values of $a$.

	\begin{remark}\label{remark2} 
		If the support of $P$ contains infinitely many elements, then the $n$th unconstrained quantization error is strictly decreasing. This fact is not true in constrained quantization, i.e., the $n$th constrained quantization error for a Borel probability measure can eventually remain constant as can be seen from Subsection~\ref{subdia}. 
	\end{remark}
	
	\begin{remark}
		By Remark~\ref{remark1} and Remark~\ref{remark2}, we can conclude that there are some properties that are true in unconstrained quantization, but are not true in constrained quantization. These motivate us to adopt more general definitions of quantization dimension and quantization coefficient, as given below, which are meaningful both in constrained and unconstrained scenarios. 
	\end{remark} 
	\begin{defi1} \label{defi001} 
		Let $P$ be a Borel probability measure on $\D R^k$ equipped with a metric $d$, and let $r \in (0,\infty)$. Let $V_{n, r}(P)$ be the $n$th constrained quantization error of order $r$ for a given $S$ that occurs in \eqref{Vr}.  Let $V_{n, r}(P)$ be a strictly decreasing sequence. 
		Then, it converges to its exact lower bound, which is a nonnegative constant. Set \[
		V_{\infty, r}(P):=\lim_{n\to \infty} V_{n, r}(P).\]
		Then, $(V_{n, r}(P)-V_{\infty, r}(P))$ is a strictly decreasing sequence of positive numbers such that \[\mathop{\lim}\limits_{n\to \infty}(V_{n, r}(P)-V_{\infty, r}(P))=0.\]
		Write
		\begin{align}
			\left\{ \begin{array}{ll}\label{eq555}
				\ul D_r(P):=\mathop{\liminf}\limits_{n\to \infty}  \frac{r\log n}{-\log (V_{n, r}(P)-V_{\infty, r}(P))}, \te{ and }\\
				\ol D_r(P):=\mathop{\limsup}\limits_{n\to \infty} \frac{r\log n}{-\log (V_{n, r}(P)-V_{\infty, r}(P))}.
			\end{array}\right.\end{align}
		$\ul D_r(P)$ and $\ol D_r(P)$ are called the \tit{lower} and the \tit{upper constrained quantization dimensions} of the probability measure $P$ of order $r$ with respect to the constraint $S$, respectively. If $\ul D_r (P)=\ol D_r (P)$, the common value is called the \tit{constrained quantization dimension} of $P$ of order $r$ with respect to the constraint $S$ and is denoted by $D_r(P)$. The constrained quantization dimension measures the speed how fast the specified measure of the constrained quantization error converges as $n$ tends to infinity. A higher constrained quantization dimension suggests a faster convergence of the $n$th constrained quantization error.  
		For any $\gk>0$, the two numbers
		\[\liminf_{n\to \infty} n^{\frac r \gk}  (V_{n, r}(P)-V_{\infty, r}(P)) \te{ and } \limsup_{n\to \infty}  n^{\frac r \gk}(V_{n, r}(P)-V_{\infty, r}(P))\] are, respectively, called the \tit{$\gk$-dimensional lower} and the \tit{upper} constrained quantization coefficients for $P$ of order $r$ with respect to the constraint $S$, respectively. If the \tit{$\gk$-dimensional lower} and the \tit{upper constrained quantization coefficients} for $P$ exist and are equal, then we call it the \tit{$\gk$-dimensional constrained quantization coefficient} for $P$ of order $r$ with respect to the constraint $S$.
	\end{defi1}

\begin{remark}
While Proposition~\ref{prop0} ensures that $\{V_{n,r}(P)\}$ is decreasing and bounded below (and thus convergent), in Definition~\ref{defi001}  we additionally assume that $\{V_{n,r}(P)\}$ is \emph{strictly} decreasing. This stronger condition is imposed to guarantee that 
\[
V_{n,r}(P) - V_{\infty,r}(P) \neq 0 \quad \text{for all finite } n,
\]
which is essential for the meaningful definition of the constrained quantization dimension and coefficient. Without strict monotonicity, the sequence may stabilize at a finite stage, causing the denominator in the associated scaling limits to vanish and leading to degeneracy in the asymptotic characterization.

\end{remark} 
	
	The following proposition is a generalized version of \cite[Propoisition~11.3]{GL}. 
	
	\begin{prop} \label{constrainedcoefficient}
		Let $P$ be a Borel probability measure, and $\underline{D}_r(P) \te{ and } \overline{D}_r(P)$ be the lower and the upper constrained quantization dimensions, respectively.   
		\begin{enumerate}
			\item If $0\leq s< \underline{D}_r(P) <t$, then
			\[\lim_{n\to \infty} n^{\frac rs}  (V_{n, r}(P)-V_{\infty, r}(P))=+\infty \te{ and } \liminf_{n\to \infty} n^{\frac rt}  (V_{n, r}(P)-V_{\infty, r}(P))=0.\]
			\item  If $0\leq s< \overline{D}_r(P) <t$, then
			\[\limsup_{n\to \infty} n^{\frac rs}  (V_{n, r}(P)-V_{\infty, r}(P))=+\infty \te{ and } \lim_{n\to \infty} n^{\frac rt}  (V_{n, r}(P)-V_{\infty, r}(P))=0.\]
		\end{enumerate}
	\end{prop} 
	\begin{proof}
		Let us first prove $(1)$. Let $0 \leq s < \underline{D}_r(P)$. Choose $s' \in (s, \underline{D}_r(P))$. Then, there exists an $n_0\in \D N$ with 
		\[(V_{n,r}(P) - V_{\infty,r}(P))<1 \te{ and } \frac{r\log n}{-\log (V_{n,r}(P) - V_{\infty,r}(P))}>s'\]
		for all $n\geq n_0$. This implies that 
		\[n^r \left( V_{n,r}(P) - V_{\infty,r}(P) \right)^{s'}>1,\]
		and so, 
		\[n^r \left( V_{n,r}(P) - V_{\infty,r}(P) \right)^{s}>\left( V_{n,r}(P) - V_{\infty,r}(P) \right)^{s-s'}\]
		for all $n\geq n_0$. Hence, 
		\[\lim_{n\to \infty} n^r \left( V_{n,r}(P) - V_{\infty,r}(P) \right)^{s}=+\infty, \te{ i.e., }  \lim_{n\to \infty} n^{\frac rs}  (V_{n, r}(P)-V_{\infty, r}(P))=+\infty.\]
		For $\underline{D}_r(P) <t$, there is a $t'\in (\underline{D}_r(P), t)$ and a subsequence $\left( V_{n_k,r}(P) - V_{\infty,r}(P) \right)$ with 
		\[\left(V_{n_k,r}(P) - V_{\infty,r}(P) \right)<1 \te{ and } 	\frac{r \log n_k}{-\log(V_{n_k,r}(P) - V_{\infty,r}(P))}\leq t'\]
		for all $k\in \D N$. This implies that 
		\[n_k^r \left( V_{n_k,r}(P) - V_{\infty,r}(P) \right)^{t'}\leq 1\]
		and so, \[n_k^r \left( V_{n_k,r}(P) - V_{\infty,r}(P) \right)^{t}\leq \left( V_{n_k,r}(P) - V_{\infty,r}(P) \right)^{t-t'}.\]
		Recall that $\lim\limits_{k\to \infty} \left( V_{n_k,r}(P) - V_{\infty,r}(P) \right)^{t-t'}=0$, and hence 
		\[\liminf_{n\to \infty} n^r \left( V_{n,r}(P) - V_{\infty,r}(P) \right)^{t}\leq \lim_{k\to\infty} n_k^r \left( V_{n_k,r}(P) - V_{\infty,r}(P) \right)^{t}=0,\]
		yielding \[\liminf_{n\to \infty} n^{\frac rt}  (V_{n, r}(P)-V_{\infty, r}(P))=0.\]
		Thus, the proof of $(1)$ is concluded. Proceeding in the similar way, we can prove $(2)$. Thus, the proof of the proposition is obtained. 
	\end{proof}

	The following corollary is a direct consequences of Proposition \ref{constrainedcoefficient}.
	\begin{cor}
		If the $\gk$-dimensional lower and the upper constrained quantization coefficients for $P$ are finite and positive, then the constrained quantization dimension $D_r(P)$ of $P$ exists and $\gk$ equals $D_r(P)$.
	\end{cor}
	Let $V_{n, 2}(P)$ be the $n$th constrained quantization error of order $2$. Then, \\
	\noindent \eqref{eq56} implies that
	\begin{equation} \label{eq59}\lim_{n\to\infty} \frac{2\log n}{-\log (V_{n, 2}(P)-V_{\infty, 2}(P))}=1  \te{ and } \lim_{n\to\infty}n^2(V_{n, 2}(P)-V_{\infty, 2}(P))=\frac{1}{24},
	\end{equation}
	\eqref{eq57} implies that
	\begin{equation} \label{eq60}\lim_{n\to\infty} \frac{2\log n}{-\log (V_{n, 2}(P)-V_{\infty, 2}(P))}=1  \te{ and } \lim_{n\to\infty}n^2(V_{n, 2}(P)-V_{\infty, 2}(P))=\frac{2744}{10125},
	\end{equation}
	\eqref{eq571} implies that
	\begin{equation} \label{eq601}\lim_{n\to\infty} \frac{2\log n}{-\log (V_{n, 2}(P)-V_{\infty, 2}(P))}=2  \te{ and } \lim_{n\to\infty}n(V_{n, 2}(P)-V_{\infty, 2}(P))=3,
	\end{equation}
	and  \eqref{eq58} implies that
	\begin{equation} \label{eq63}
		\lim_{n\to\infty} \frac{2\log n}{-\log (V_{n, 2}(P)-V_{\infty, 2}(P))}=1  \te{ and } \lim_{n\to\infty}n^2(V_{n, 2}(P)-V_{\infty, 2}(P))=\frac{\pi ^2 a}{3}.
	\end{equation}

	\subsection{Observations and Conclusions}
	\begin{enumerate}
		
		\item In unconstrained quantization the elements in an optimal set, for a Borel probability measure $P$, are the conditional expectations in their own Voronoi regions. This fact is not true in constrained quantization, for example, for the probability measure $P$, defined in Corollary~\ref{cor1}, the optimal set of two-points is obtained as $\set{(\frac 18, \frac {1}{8}\sqrt 3), (\frac 38, \frac {3}{8}\sqrt 3)}$, and the set of conditional expectations of the Voronoi regions is $\set{(\frac 12, 0), (\frac 32, 0)}$, i.e., the two sets are different.
		
		\item In unconstrained quantization if the support of $P$ contains at least $n$ elements, then an optimal set of $n$-points contains exactly $n$ elements. This fact is not true in constrained quantization. For example, from Subsection~\ref{subdia}, we see that if a Borel probability measure $P$ on $\D R^2$ has support the diameter of a circle and the constraint $S$ is the circle, then the optimal sets of $n$-points containing exactly $n$ elements exist only for $n=1$ and $n=2$, and they do not exist for any $n\geq 3$, though the support has infinitely many elements.

		\item  In unconstrained quantization, the quantization dimension of an absolutely continuous probability measure exists and equals the Euclidean dimension of the support of $P$. This fact is not true in constrained quantization, as can be seen from the expressions \eqref{eq59}, \eqref{eq60}, and \eqref{eq601}. Each of the probability measures has support the closed interval $[0, 2]$ on a line, but the quantization dimensions are different, i.e., the quantization dimension in constrained quantization depends on the constraint $S$ that occurs in the definition of constrained quantization error. The quantization dimension, in the case of unconstrained quantization, if it exists, measures the speed how fast the specified measure of the error goes to zero as $n$ tends to infinity, on the other hand, in the case of constrained quantization, if it exists, measures the speed how fast the specified measure of the error converges as $n$ tends to infinity.
		
		\item In unconstrained quantization, the quantization coefficient for an absolutely continuous probability measure exists as a unique finite positive number. In constrained quantization, the quantization coefficient for an absolutely continuous probability measure also exists, but it is not unique, and can be any nonnegative number as can be seen from the expressions of quantization coefficients in \eqref{eq59}, \eqref{eq60}, \eqref{eq601}, and \eqref{eq63}, i.e., the quantization coefficient in constrained quantization depends on the constraint $S$ that occurs in the definition of constrained quantization error.
\end{enumerate}

\section{Future Directions} \label{sec6}
The results in this paper suggest several concrete open problems and directions for future research in constrained quantization.

A fundamental question is to characterize how the geometry of the constraint set \( S \) influences the constrained quantization dimension and quantization coefficient. In particular, it would be of interest to determine how intrinsic properties of \( S \), such as its dimension, curvature, smoothness, convexity, or fractal structure, affect asymptotic quantization rates. For example, one may ask whether the constrained quantization dimension coincides with the Hausdorff or Minkowski dimension of \( S \) under suitable regularity conditions, and how curvature or boundary effects modify optimal point configurations.

Another important direction is to extend constrained quantization theory beyond the Euclidean setting. This includes developing an analogous framework for Riemannian manifolds equipped with geodesic distance, as well as for more general metric or non-Euclidean spaces. Such an extension raises natural questions about whether quantization behavior is governed primarily by intrinsic geometry and how metric distortion or curvature influences the existence, structure, and asymptotic properties of optimal sets.

A further open problem concerns existence, uniqueness, and stability of constrained optimal sets. While existence of optimal one-point sets holds in broad generality, for \( n \geq 2 \) constrained optimal sets may fail to exist or may exhibit degeneracies depending on the geometry of \( S \). It would be valuable to establish geometric or measure-theoretic conditions ensuring the existence of \( n \) distinct optimal points, as well as to study stability of optimal configurations under perturbations of the probability measure \( P \) or small deformations of the constraint set \( S \).

From a computational perspective, although we derive explicit optimal configurations for several constrained geometries, the systematic design of numerical optimization algorithms and their computational implementation constitutes an important direction for future research. Developing efficient computational methods for constrained quantization could facilitate large-scale experimentation, visualization of optimal structures, and practical applications in engineering and data-driven optimization.

We believe that addressing these questions will deepen the understanding of how geometric constraints shape quantization behavior and will broaden the applicability of constrained quantization to geometric optimization, signal processing, and related areas.


 \subsection*{Acknowledgement}

The authors are grateful to \textit{Professor Carl P. Dettmann} of the University of Bristol, UK, for his valuable comments and suggestions, which contributed to improving this manuscript. 
The first author also expresses sincere gratitude to her supervisor, \textit{Professor Tanmoy Som} of IIT (BHU), Varanasi, India, for his guidance, encouragement, and continuous support during the preparation of this work. 
The authors further thank the anonymous referees for their careful reading of the manuscript and their constructive feedback, which helped strengthen the clarity and quality of the paper.

\section*{Declaration}
							
							\noindent
							\textbf{Conflicts of interest.} We do not have any conflict of interest.\\
							\\
							\noindent
							\textbf{Data availability:} No data were used to support this study.\\
							\\
							\noindent
							\textbf{Code availability:} Not applicable\\
							\\
							\noindent
							\textbf{Authors' contributions:} Each author contributed equally to this manuscript.



\end{document}